\documentclass[10pt,a4paper,english,reqno]{amsart}
\usepackage{amsmath}
\usepackage{a4}
\usepackage{amssymb}
\usepackage{amsthm}
\usepackage{bbm}
\usepackage{color}
\usepackage{bm}
\usepackage{dsfont}
\usepackage[T1]{fontenc}

\usepackage[utf8]{inputenc}
\usepackage[scaled=0.86]{helvet}

\usepackage[english]{babel}
\usepackage{latexsym}
\usepackage{mathtools}
\usepackage{cite}
\usepackage{float}
\usepackage{placeins}
\usepackage{caption}

\usepackage{enumerate}

\usepackage{a4wide}
\usepackage[arrow, matrix, curve]{xy}

\usepackage{cases}
\usepackage[unicode=true,pdfusetitle,
 bookmarks=true,bookmarksnumbered=false,bookmarksopen=false,
 breaklinks=false,pdfborder={0 0 1},colorlinks=false]
 {hyperref}

\newcommand{\N}{\mathbb N}
\newcommand{\R}{\mathbb R}

 \newcommand{\X}{X}

\makeatletter

\pdfpageheight\paperheight
\pdfpagewidth\paperwidth

\theoremstyle{plain}
\newtheorem{thm}{Theorem}[section] 
\theoremstyle{definition}
\newtheorem{Def}[thm]{\protect\definitionname}

\theoremstyle{plain}

\theoremstyle{plain}
\newtheorem{Prop}[thm]{Proposition}
\newtheorem{lemma}[thm]{Lemma}

\theoremstyle{definition}

\theoremstyle{remark}
\newtheorem{Rem}[thm]{Remark}

\newcommand\upperleft[3]{\prescript{#1}{}{\mathrlap{\smash{#2#3}}\phantom{#3}}}

\DeclareMathOperator*{\esssup}{esssup}
\DeclareMathOperator*{\essinf}{essinf}
\DeclareMathOperator{\id}{id}
\DeclareMathOperator{\Image}{Im}

\DeclareMathOperator{\spann}{span}

\DeclareMathOperator{\osc}{osc}

\renewcommand{\emptyset}{\varnothing}
\renewcommand{\bar}{\overline}
\usepackage{amssymb}
\usepackage{mathtools}
\makeatother

  \providecommand{\definitionname}{Definition}

\begin{document}
\title[Intermediately trimmed strong laws for Birkhoff sums on subshifts of finite type]
{Intermediately trimmed strong laws for Birkhoff sums on subshifts of finite type}

\author[Kesseb\"ohmer]{Marc Kesseb\"ohmer}
  \address{Universit\"at Bremen, Fachbereich 3 -- Mathematik und Informatik, Bibliothekstr. 1, 28359 Bremen, Germany}
  \email{\href{mailto:mhk@math.uni-bremen.de}{mhk@math.uni-bremen.de}}
\author[Schindler]{Tanja Schindler}
\address{Australian National University, Research School Finance, Actuarial Studies and Statistics, 26C Kingsley St,
Acton ACT 2601, Australia}
  \email{\href{mailto:tanja.schindler@anu.edu.au}{tanja.schindler@anu.edu.au}}

\date{\today}
\thanks{This research was supported  by the German Research Foundation (DFG) grant {\em Renewal Theory and Statistics of Rare Events in Infinite Ergodic Theory} (Gesch\"aftszeichen KE 1440/2-1).}

\begin{abstract}
We prove strong laws of large numbers under intermediate trimming for 
Birkhoff sums over subshifts of finite type.
This gives another application of a previous trimming result only proven for interval maps. 
In case of Markov measures we give a further example of St.\ Petersburg type distribution functions. 
To prove these statements we introduce the space of quasi-H\"older continuous functions 
for subshifts of finite type.
\end{abstract}
\maketitle

\section{Introduction and statement of main results}
If we consider an ergodic dynamical system $\left(\Omega,\mathcal{A},T, \mu\right)$ with $\mu$ a probability measure  
and a stochastic process given by the Birkhoff sum $\mathsf{S}_n\chi\coloneqq\sum_{k=1}^n\chi\circ  T^{k-1}$ with $\mathsf{S}_0\chi=0$ for some measurable function $\chi:\Omega\to \mathbb{R}_{\geq 0}$,
then with respect to a strong laws of large numbers there is a crucial difference between $\int\chi\mathrm{d}\mu$ being finite or not.
In the finite case we obtain by Birkhoff's ergodic theorem that $\mu$-almost surely (a.s.)  
$\lim_{n\to\infty}\mathsf{S}_n\chi/n=\int\chi\mathrm{d}\mu$,
i.e.\  the strong law of large numbers is fulfilled,
whereas for the case that $\chi$ is non-integrable, Aaronson ruled out the possibility of a strong law of large numbers, see \cite{aaronson_ergodic_1977}.

However, in certain cases it is possible to obtain a strong law of large numbers 
after deleting a number of the largest summands from the partial $n$-sums. More precisely,
for each $n\in\mathbb{N}$ we chose a permutation $\pi\in\mathcal{S}_{n}$
of $\left\{ 0,\ldots,n-1\right\} $ with $\chi\circ  T^{\pi\left(1\right)}\geq \chi\circ  T^{\pi\left(2\right)}\geq\ldots\geq \chi\circ  T^{\pi\left(n\right)}$
and for given $b\in\mathbb{N}_{0}$ we define 
\begin{align*}
\mathsf{S}_{n}^{b}\chi & \coloneqq\sum_{k=b}^{n-1}\chi\circ  T^{\pi\left(k\right)}.
\end{align*}

In \cite{kessebohmer_strong_2017} the authors considered a general setting of 
dynamical systems $(\Omega, \mathcal{A}, \mu, T)$
obeying a set of conditions given in Property $\mathfrak{D}$, see Definition \ref{def: Prop D}.
One key assumption is that the transfer operator (see \eqref{hat CYRI}) fulfills 
a spectral gap property on $\mathcal{F}$, 
a subset of the measurable functions forming a Banach algebra with respect to a norm $\left\|\cdot\right\|$.
For such systems the authors proved \emph{intermediately trimmed strong laws},
i.e.\ the existence of a sequence of natural numbers $(b_n)$ tending to infinity
with $b_n=o(n)$
and a norming sequence $(d_n)$
such that $\lim_{n\to\infty}\mathsf{S}_{n}^{b_{n}}\chi/d_n=1$ a.s.
One interesting case
is the example of regularly varying tail distributions for which the same trimming sequence can be chosen as in the i.i.d.\ case.
The key example studied in \cite{kessebohmer_strong_2017} are piecewise expanding interval maps.

It is the aim of the present paper to adapt these results 
to subshifts of finite type.
It turns out that in contrary to the example of piecewise expanding interval maps it is not immediately clear how to apply the results from \cite{kessebohmer_strong_2017} 
to subshifts of finite type
as the usually considered space of Lipschitz continuous functions does not fulfill all required properties of Property $\mathfrak{D}$.
In particular, Property $\mathfrak{D}$ requires $\left\|\mathbbm{1}_{\left\{\chi>\ell\right\}}\right\|\leq K$ uniformly in $\ell$
which can not be fulfilled for an unbounded potential $\chi$ with respect to the Lipschitz norm,
see Remark \ref{rem: equiv norms}.

Instead we consider the Banach space of quasi-H\"older continuous functions which is larger than the space of Lipschitz continuous functions but still obeys a spectral gap property. 
A similar Banach space was first considered by Blank, see \cite[Chapter 2.3]{blank_discreteness_1997}, and Saussol, see \cite{saussol_absolutely_2000}, in the context of multidimensional expanding maps. 

Furthermore, we prove some new limit theorems not given in \cite{kessebohmer_strong_2017} 
which we consider as particularly interesting for the application to observables on a subshift with a Gibbs-Markov measure, see Section \ref{soft Sbn}.

As a side result we obtain a limit theorem for sums of the truncated random variables.
Namely, for an observable $\chi:\Omega\to\mathbb{R}_{\geq 0}$ and a real valued sequence $\left(f_{n}\right)_{n\in\mathbb{N}}$
we consider the truncated sum $\mathsf{T}_n^{f_n}\chi\coloneqq \sum_{k=1}^n\left(\chi \cdot\mathbbm{1}_{\{\chi\leq f_n\}}\right)\circ T^{k-1}$.
These results are
not only valid for the considered setting of subshifts of finite type but also for other dynamical systems fulfilling a spectral gap property
and for i.i.d.\ random variables. 
This also improves a limit theorem for St.\ Petersburg games given in \cite{gyorfi_rate_2011} and \cite{nakata_limit_2015} for i.i.d.\ random variables
with stronger conditions on the truncation sequence, see Theorem \ref{thm: st peterburg} and Remark \ref{Rem: St Peterburg}.
\medskip

It is worth mentioning that proving limit theorems under trimming for non-integrable (independent and dependent) random variables
has a long tradition and particularly for the i.i.d.\ case there is a vast literature. 
Here we will only mention results stating a generalized strong law of large numbers. 

The first considered limit laws were \emph{lightly trimmed strong laws}, 
i.e.\ the existence of $b\in\mathbb{N}$ independent of $n$ and a norming sequence $(d_n)$ such that 
$\lim_{n\to\infty}\mathsf{S}_n^b/d_n=1$ holds almost surely. 

Mori developed
 general conditions
for lightly trimmed strong laws
for i.i.d.\ random variables, see \cite{mori_strong_1976}, \cite{mori_stability_1977}. 
These results have been generalized by Kesten and Maller, see \cite{maller_relative_1984}, and 
\cite{kesten_ratios_1992}.

It became obvious by a result by Kesten, see \cite{kesten_effect_1995},
that light trimming is not always enough to prove strong laws of large numbers, 
as in particular weak laws of large numbers for i.i.d.\ random variables 
do not change under light trimming.
One example for which no weak law of large numbers hold are regularly varying tail distributions with exponent $\alpha\in(-1,0)$.
For such distribution functions an intermediately trimmed strong law
can be easily deduced from results by Haeusler and Mason, see \cite{haeusler_laws_1987}, and a lower bound for the trimming sequence $(b_n)$ can be derived from
a result by Haeusler, see \cite{haeusler_nonstandard_1993}. 
In \cite{haeuler_asymptotic_1991} and \cite{kessebohmer_strong_2016} intermediately trimmed strong laws for other distribution functions are given.

The history for trimmed strong laws in the dynamical systems setting follows a similar line.
One of the first investigated examples for the case of dynamical systems is the unique continued
fraction expansion. Diamond and Vaaler showed in \cite{diamond_estimates_1986} a lightly trimmed strong law
for the digits of the continued fraction expansion.
Aaronson and Nakada extended the afore mentioned results by Mori to $\bm{\psi}$-mixing random variables
in \cite{aaronson_trimmed_2003}, i.e.\ they gave sufficient conditions for a lightly trimmed strong law to hold.
In \cite{haynes_quantitative_2014} Haynes gave a quantitative strong law of large numbers under light trimming
for certain classes of dynamical systems. 
The results in \cite{aaronson_trimmed_2003} and \cite{haynes_quantitative_2014} are also 
applicable for subshifts of finite type obeying some additional conditions.

However, for certain types of dynamical systems light trimming is not enough 
even though a lightly trimmed strong law would hold in the i.i.d.\ case for random variables with the same distribution function, 
see \cite{aaronson_trimmed_2003} and \cite{haynes_quantitative_2014}; see also \cite{schindler_observables_2018} for the corresponding
intermediately trimmed strong law for such a system.

In contrast to the lightly trimmed case an intermediately trimmed strong law 
for regularly varying tail distributions with exponent $\alpha\in(-1,0)$
holds with the same trimming sequence $(b_n)$ and norming sequence $(d_n)$
for both i.i.d.\ random variables and a large class of dynamical systems,
see \cite{kessebohmer_strong_2017}.
As noted above we will prove that a large class of observables 
on subshifts of finite type can also fulfill these properties 
and the same intermediately trimmed strong law holds.

\subsection{The space of quasi-H\"older continuous functions}
In order to state our main theorem we first introduce the space of quasi-H\"older continuous functions on the subshift of finite type
\begin{align*}
\X \coloneqq\left\{x=\left(x_n\right)_{n\in\mathbb{N}_0}\colon x_n\in \mathbb{A},\text{ } A\left(x_n,x_{n+1}\right)=1\right\},
\end{align*}
where $\mathbb{A}$ denotes a finite alphabet and 
$A\in\left\{0,1\right\}^{\mathbb{A}\times \mathbb{A}}$ is an irreducible and aperiodic matrix (see Definition \ref{irred aperiod}).
For the following we denote by $\mathbb{A}_n$ the set of all admissible sequences of length $n$
and for $\mathsf{A}\in\mathbb{A}_n$ we let $\left[\mathsf{A}\right]\subset \X$ denote the cylinder set determined by $\mathsf{A}$ 
and let $\X$ represent the cylinder set of the empty word.
The $\sigma$-algebra generated by the set of cylinder sets will be denoted by $\mathcal{B}$.
With $\sigma:\X\to \X$ we denote the left shift and
  we define a metric $d_1$ on $\X$ by
$d_1\left(x,y\right)\coloneqq \theta^{k}$, for some $\theta\in\left(0,1\right)$ where $k\in\mathbb{N}_0$ is the largest integer 
such that $x_i=y_i$ for all $i\leq k$.
Next we fix a probability measure $\mu$ 
that is a  $g$-measure 
such that  the  corresponding $g$-function (see Definition   \ref{def: g func meas}) is given by 
$\exp(f)$ for $f\colon \X \to\mathbb{R}_{>0}$ being Lipschitz continuous with respect to the metric $d_1$.
Note that a $g$-measure is always $\sigma$-invariant, atomless and assigns positive measure to every non-empty cylinder set.

For a measurable function $h\colon\X \to\mathbb{R}_{\geq 0}$ the oscillation on $C\subset \X $ is given by
\begin{align*}
 \osc\left( h ,C\right)\coloneqq\esssup_{x\in C} h \left(x\right)-\essinf_{x\in C} h \left(x\right)
\end{align*}
and we set $\osc\left( h ,\emptyset\right)\coloneqq0$.
For the fixed  measure $\mu$ we define the metric  $d_2$ on $\X$ given by
\[
  d_2(x,y)\coloneqq \inf\{\mu(C)\colon x,y\in C, C\subset \X  \mbox{  cylinder set}\}
\] 
and we let $B\left(\epsilon,x\right)$ denote the $\epsilon$-ball around $x$ with respect to this metric.
Then for fixed $\epsilon_0\in \left(0,1\right)$ 
we define
\begin{align*}
\left| h \right|_{\epsilon_0}\coloneqq\sup_{0<\epsilon\leq\epsilon_0}\frac{\int \osc\left( h , B\left(\epsilon,x\right)\right) \mathrm{d}\mu}{\epsilon} 
\end{align*}
to obtain the norm
\begin{align*}
\left\| h \right\|_{\epsilon_0}\coloneqq \left| h \right|_{\infty}+\left| h \right|_{\epsilon_0}.
\end{align*}
We will show in Lemma \ref{lem: H Banach algebra} that 
$\left\|\cdot \right\|_{\epsilon_0}$ is indeed a norm.
With this at hand we can define the space of {\em quasi-H\"older continuous functions}
\begin{align*}
\mathsf{H}\coloneqq\left\{ h \in \mathcal{L}^{\infty}\left(\X ,\mathbb{R}\right)\colon \left| h \right|_{\epsilon_0}<\infty\right\}.
\end{align*}
We remark here that the norm $\|\cdot\|_{\epsilon_0}$ depends on $\epsilon_0$, 
but $\mathsf{H}$ is independent of the choice of $\epsilon_0$.

Given these definitions we are able to state the setting for our main theorem.
\begin{Def}
For a $\mathcal{B}$-measurable function $\chi\colon \X \to \mathbb{R}_{\geq 0}$ we set $\upperleft{\ell}{}{\chi}\coloneqq\chi\cdot\mathbbm{1}_{\left\{\chi\leq \ell\right\}}$ and say that $\left(\X , \mathcal{B}, \mu,\chi\right)$ fulfills Property $\mathfrak{F}$ if 
the following conditions hold:  
\begin{itemize}
 \item There exists $K_1>0$ and $\epsilon_0>0$ such that for all $\ell>0$
 \begin{align}
  \left|\upperleft{\ell}{}{\chi}\right|_{\epsilon_0}\leq K_1\ell.\label{eq: cond chi 1}
 \end{align}
\item There exists $K_2>0$ and $\epsilon_0'>0$ such that for all $\ell>0$
 \begin{align}
  \left|\mathbbm{1}_{\{\chi\geq \ell\}}\right|_{\epsilon_0'}\leq K_2.\label{eq: cond chi 2}
 \end{align}
\end{itemize}
\end{Def}

\begin{Rem}\label{rem: equiv norms}
In \cite{kessebohmer_strong_2017} we impose Property $\mathfrak{D}$
(see Definition \ref{def: Prop D})
as a sufficient condition for a  trimmed strong law to hold, which in particular 
requires  that there exists $K_2>0$ such that for all $\ell>0$ we have
$\left\|\mathbbm{1}_{\left\{\chi>\ell\right\}}\right\|\leq K_2$, see \eqref{C 2}.
Now, the space of Lipschitz continuous functions (with respect to the metric $d_1$)
is given by
$F_{\theta}\left(\X \right)\coloneqq\left\{f:\X \to\mathbb{R}_{\geq0}\colon \left|f\right|_{\theta}<\infty\right\}$,
where 
\begin{align*}
\left|f\right|_{\theta}\coloneqq\sup_{x,y\in\X}\frac{\left|f\left(x\right)-f\left(y\right)\right|}{d_1\left(x,y\right)}
\end{align*}
and is equipped with the norm
$\left\|f\right\|_{\theta}\coloneqq\left|f\right|_{\infty}+\left|f\right|_{\theta}$.
From this it becomes apparent that  we can not find an unbounded observable $\chi$ such that 
Property $\mathfrak{D}$ is fulfilled for 
$\left(\X ,\mathcal{B},\mu,\sigma,F_{\theta},\left\|\cdot\right\|_{\theta},\chi\right)$.
Indeed, if $\chi$ is unbounded, then, for all $n>0$ there exists $\ell>0$ and $x,y\in\X$ 
with $d_1\left(x,y\right)\leq \theta^n$, 
$\mathbbm{1}_{\{\chi> \ell\}}\left(x\right)=1$,
and $\mathbbm{1}_{\{\chi> \ell\}}\left(y\right)=0$.
Hence,  $\left\|\mathbbm{1}_{\{\chi> \ell\}}\right\|_{\theta}\geq \theta^{-n}$. 
For this reason it turns out that the quasi-H\"older continuous functions witnessing  only one pole are   more appropriate observables to allow \eqref{C 2} to hold.
\end{Rem}

\subsection{Main theorem}\label{subshifts}
Before stating our main theorem we first define the notion of regular and slow variation. 
A function $L$ is called slowly varying if for every $c > 0$
we have $L (cx) \sim L (x)$.
Here, $u (x) \sim w (x)$ means that $u$ is asymptotic to $w$
at infinity, i.e.\ $\lim_{x\to\infty}u (x) /w (x) = 1$.
A function $g$ is called regularly varying with index $\alpha$ if it can be written as 
$g(x)=L(x)\cdot x^{\alpha}$ with $L$ being slowly varying.
For $L$ being slowly varying we denote by $L^{\#}$
a de Bruijn conjugate of $L$, i.e.\ 
a slowly varying function satisfying 
\begin{align*}
\lim_{x\rightarrow\infty}L\left(x\right)\cdot L^{\#}\left(xL\left(x\right)\right)  =1=
\lim_{x\rightarrow\infty}L^{\#}\left(x\right)\cdot L\left(xL^{\#}\left(x\right)\right).
\end{align*}
For more details  see \cite[Section 1.5.7 and Appendix 5]{bingham_regular_1987}.
Furthermore, we set
\[
\Psi\coloneqq\left\{ u:\mathbb{N}\rightarrow\mathbb{R}_{>0}\colon\sum_{n=1}^{\infty}\frac{1}{u\left(n\right)}<\infty\right\} .
\]
\begin{thm}\label{thm: main thm reg var} 
Let $\left(\X , \mathcal{B}, \mu,\chi\right)$
fulfill Property $\mathfrak{F}$
and let additionally $\chi$ be such that 
$\mu\left(\chi>x\right)=L\left(x\right)/x^{\alpha}$ with $L$ a slowly varying function and $0<\alpha<1$. 
Further, let $\left(b_{n}\right)_{n\in\mathbb{N}}$ be a sequence
of natural numbers tending to infinity with $b_{n}=o\left(n\right)$.
If there exists $\psi\in\Psi$ such that 
\begin{align*}
\lim_{n\rightarrow\infty}\frac{b_{n}}{\log\psi\left(\left\lfloor \log n\right\rfloor \right)}=\infty,
\end{align*}
then there exists a positive valued sequence $\left(d_n\right)_{n\in\mathbb{N}}$ such that 
\begin{align}
\lim_{n\rightarrow\infty}\frac{\mathsf{S}_n^{b_n}\chi}{d_n}=1\text{ a.s.}\label{eq: lim Snbnchi}
\end{align}
and $\left(d_n\right)$ fulfills 
\begin{align}
d_{n}\sim\frac{\alpha}{1-\alpha}\cdot n^{1/\alpha}\cdot b_{n}^{1-1/\alpha}\cdot \left(L^{1/\alpha}\right)^{\#}\left(\left({n}/{b_{n}}\right)^{1/\alpha}\right).\label{bn psi exp}
\end{align}
\end{thm}
See \cite[Remark 1.8 and Remark 1.9]{kessebohmer_strong_2017} for remarks on this theorem in the general setting.

\begin{Rem}
Note that in \cite{kessebohmer_strong_2017} more strong laws under trimming are provided for tuples $\left(\Omega, \mathcal{A}, T , \mu,\mathcal{F},\left\|\cdot\right\|,\chi\right)$ fulfilling Property $\mathfrak{D}$  (see   Definition \ref{def: Prop D}).
Our approach is to prove that $\left(\X , \mathcal{B}, \mu,\chi\right)$
fulfilling Property $\mathfrak{F}$
implies the existence of $\epsilon_0\in\left(0,1\right)$
such that $\left(\X , \mathcal{B},\sigma, \mu,\mathsf{H},\left\|\cdot\right\|_{\epsilon_0},\chi\right)$ fulfills Property $\mathfrak{D}$, see Lemma \ref{lem: Prop E implies Prop D}.
Hence, all the statements given in \cite{kessebohmer_strong_2017} also hold for the tuple
$\left(\X , \mathcal{B},\sigma, \mu,\mathcal{F},\left\|\cdot\right\|_{\epsilon_0},\chi\right)$.
For brevity we will not restate these results.
\end{Rem}

\subsection{Trimming statements for Markov systems}\label{soft Sbn}
\begin{thm}\label{thm: Markov}
Let $\X $ be a one-sided subshift of finite type and $\mathbb{A}\coloneqq\left\{\mathfrak{a}_1,\ldots,\mathfrak{a}_k\right\}$, $k\geq 2$, its alphabet. 
Let  $\mu$ be a stationary Markov measure compatible  with an irreducible and aperiodic matrix $A\in\left\{0,1\right\}^{k\times k}$ (see Definition \ref{irred aperiod} and Definition \ref{Def markov measure}) such that $q\coloneqq\mu\left(x_n=\mathfrak{a}_1\lvert x_{n-1}=\mathfrak{a}_1\right)> 0$ and $R\coloneqq\mu\left(x_1=\mathfrak{a}_1\right)/q$.
For $\eta>1/q$ let $\chi:\X \to\mathbb{R}_{\geq 0}$ be given by
\begin{align*}
 \chi\left(\left(x_n\right)\right)\coloneqq\eta^{\min\left\{j\in\mathbb{N}\colon x_j\neq \mathfrak{a}_1\right\}-1}.
\end{align*}
Further, let $\left(b_n\right)_{n\in\mathbb{N}}$ be a sequence of natural numbers tending to infinity with $b_n=o\left(n\right)$
and assume there exists $\psi\in\Psi$ fulfilling
\begin{align*}
\lim_{n\to\infty}\frac{b_n}{\log\psi\left(\left\lfloor\log n\right\rfloor\right)}=\infty.
\end{align*}
Then there exists a sequence of constants $(d_n)$ such that
\begin{align*}
\lim_{n\to\infty}\frac{\mathsf{S}_n^{b_n}\chi}{d_n}=1\text{ a.s.}
\end{align*}
\end{thm}
\begin{Rem}
In certain cases $d_n$ can be explicitly given:
Assume that
$\left(b_n\right)$ can be written as 
$b_n=R/(1-q)\cdot q^{k_n}\cdot n+w_n$
with 
$w_n\geq \left(q^{k_n}\cdot n\right)^{1/2+\epsilon}\cdot \log \psi\left(\left\lfloor \log n\right\rfloor\right)^{1/2-\epsilon}$
and $w_n=o\left(q^{k_n}\right)$ for $\left(k_n\right)$ being a sequence of natural numbers,
then 
\begin{align*}
d_{n}\sim\frac{\eta}{q\cdot \eta-1}\cdot R^{-\log \eta/\log q}\cdot \left(1-q\right)^{1+\log \eta/\log q}\cdot  n^{-\log\eta/\log q}\cdot b_n^{1+\log\eta/\log q}.
\end{align*}

Note that there is an analogy to formula \eqref{bn psi exp} with $\alpha\coloneqq-\log q/\log\eta$ and $L=1$. However, it is not possible to apply the same method here since $\chi$ does not have a regularly varying tail distribution.
\end{Rem}

\begin{Rem}
 We will state a theorem with a St.\ Petersburg type distribution functions also in the more general setting introduced in \cite{kessebohmer_strong_2017}, see Theorem \ref{thm: asymp thm},
 i.e.\ in this setting the theorem can also be applied to piecewise expanding interval maps (and possibly other systems).
\end{Rem}

\subsection{Structure of the paper}
In Section \ref{sec: add stat} 
we state and prove 
additional limit theorems under the more general setting of Property $\mathfrak{D}$.
This is a general property for dynamical systems given in \cite{kessebohmer_strong_2017}.
We specify this property in Section \ref{subsec: gen setting}
and state an intermediately trimmed strong law for St.\ Petersburg type distribution functions in Section \ref{gen trimming}. 
In Section \ref{sec: proof main thm} we give a general approach for proving an intermediately trimmed strong law.
This is a refinement of the approach given in \cite[Section 2.1]{kessebohmer_strong_2017} 
putting extra emphasize on the appearance of ties. 
One main ingredient of this proof is a truncated limit theorem given in Section \ref{subsec: trunc sums}. 
Finally, we give the proof of the general theorem in Section \ref{subsec: proof asymp thm}
which will be the basis to prove Theorem \ref{thm: Markov}.

Section \ref{sec: quasi Hoelder} is devoted to prove certain properties of the space of quasi-H\"older continuous functions 
for subshifts of finite type. After giving some necessary definitions in Section \ref{subsec: def and rem}
we prove the main property of this space, the spectral gap property, in Section \ref{subsec: spec gap}.
In Section \ref{subsec: prop D holds} we prove Theorem \ref{thm: main thm reg var} and Theorem \ref{thm: Markov}.
The main step is to show that
 if $(\X , \mathcal{B},\mu, \chi)$ fulfills Property $\mathfrak{F}$, 
 then there exists $\epsilon_0\in (0,1)$ such that $\left(\X ,\mathcal{B},\mu,\sigma,\mathsf{H},\left\|\cdot\right\|_{\epsilon_0},\chi\right)$
 fulfills Property $\mathfrak{D}$.
This enables us to use the machinery established in \cite{kessebohmer_strong_2017}
and to prove Theorem \ref{thm: main thm reg var}.
Finally, we take a closer look at the example of Gibbs-Markov measures given in Section \ref{soft Sbn}
and show that for this setting Property $\mathfrak{F}$ holds which proves Theorem \ref{thm: Markov}.

\section{Further limit results in the general setting}\label{sec: add stat}
\subsection{General setting}\label{subsec: gen setting}
In the following we denote the spectral radius of an operator $U$ by $\rho\left(U\right)$
and give first the definition of a spectral gap.
\begin{Def}[Spectral gap]\label{def spec gap}
\index{spectral gap}
Suppose $\mathcal{F}$ is a Banach space and $U:\mathcal{F}\to\mathcal{F}$ a bounded linear operator. We say that $U$ has a spectral gap if there exists a decomposition $U=\lambda P+N$ with $\lambda\in\mathbb{C}$ and $P,N$ bounded linear operators such that 
\begin{itemize}
\item $P$ is a projection, i.e.\ $P^2=P$ and $\dim\left(\Image\left(P\right)\right)=1$\index{projection},
\item $N$ is such that $\rho\left(N\right)<\left|\lambda\right|$,
\item $P$ and $N$ are orthogonal, i.e.\ $PN=NP=0$\index{orthogonal}.
\end{itemize}
\end{Def}

Next we state the two main properties from \cite{kessebohmer_strong_2017}.
Under this setting we will prove in the next section an intermediately trimmed strong law 
for St.\ Petersburg type distribution functions.
\begin{Def}[Property $\mathfrak{C}$, {\cite[Definition 1.1]{kessebohmer_strong_2017}}]\label{def: prop C}
Let $\left(\Omega, \mathcal{A},  T, \mu\right)$ be a dynamical system with $ T$ a non-singular transformation and $\widehat{ T}:\mathcal{L}^1\to \mathcal{L}^1$ be the transfer operator of $ T$, 
i.e.\ the uniquely defined operator such that for all $f\in\mathcal{L}^1$ and $g\in\mathcal{L}^{\infty}$ we have
\begin{align}
\int \widehat{ T}f\cdot g\mathrm{d}\mu=\int f\cdot g\circ T\mathrm{d}\mu,\label{hat CYRI}
\end{align}
see e.g. \cite[Section 2.3]{MR3585883} for further details.
Furthermore, let $\mathcal{F}$ be subset of the measurable functions forming a Banach algebra with respect to the  norm $\left\|\cdot\right\|$.
We say that $\left(\Omega, \mathcal{A}, T , \mu,\mathcal{F},\left\|\cdot\right\|\right)$ has Property $\mathfrak{C}$ if the following conditions hold:
\begin{itemize}
\item
$\mu$ is a $T $-invariant, mixing probability measure.
\item
$\mathcal{F}$ contains the constant functions and for all $f\in\mathcal{F}$ we have
\begin{align*}
\left\|f\right\|\geq \left|f\right|_{\infty}.
\end{align*}
\item 
$\widehat{T}$ is a bounded linear operator with respect to $\left\|\cdot\right\|$, i.e.\ there {exists} a constant $K_0>0$ such that for all $f\in\mathcal{F}$ we have
\begin{align*}
\left\|\widehat{T}f\right\|\leq K_0\cdot\left\|f\right\|.
\end{align*}
\item 
$\widehat{T}$ has a spectral gap on $\mathcal{F}$ with respect to $\left\|\cdot\right\|$.
\end{itemize}
\end{Def}
The above mentioned property is a widely used setting for dynamical systems. 
\begin{Def}[Property $\mathfrak{D}$, {\cite[Definition 1.2]{kessebohmer_strong_2017}}]\label{def: Prop D}
We say that $\left(\Omega, \mathcal{A}, T , \mu,\mathcal{F},\left\|\cdot\right\|,\chi\right)$ has Property $\mathfrak{D}$ if the following conditions hold:
\begin{itemize}
\item $\left(\Omega, \mathcal{A}, T , \mu,\mathcal{F},\left\|\cdot\right\|\right)$ fulfills Property $\mathfrak{C}$.
\item $\chi\in\mathcal{F}^+\coloneqq\{f\in \mathcal{F}\colon f\geq 0\}$.
\item
For $\upperleft{\ell}{}{\chi}\coloneqq \chi\cdot\mathbbm{1}_{\left\{\chi\leq \ell\right\}}$ there exists
$K_1>0$ such that for all $\ell>0$, 
\begin{align}
\left\|\upperleft{\ell}{}{\chi}\right\|\leq K_1\cdot \ell.\label{C 1}
\end{align}
\item
There exists 
$K_2>0$ such that for all $\ell>0$, 
\begin{align}
\left\|\mathbbm{1}_{\left\{\chi>\ell\right\}}\right\|\leq K_2.\label{C 2} 
\end{align}
\end{itemize}
\end{Def}

\subsection{Trimming results for St.\ Petersburg type distribution functions}\label{gen trimming}
In this section we will prove theorems under the more general setting that the tuple
$\left(\Omega, \mathcal{B},  T, \mu,\mathcal{F},\left\|\cdot\right\|,\chi\right)$ fulfills Property $\mathfrak{D}$.
From Lemma \ref{lem: Prop E implies Prop D} we will see that this can immediately be applied to the setting of subshifts of finite type.

\begin{thm}\label{thm: asymp thm}
Let $\left(\Omega, \mathcal{B},  T, \mu,\mathcal{F},\left\|\cdot\right\|,\chi\right)$ 
fulfill Property $\mathfrak{D}$ and assume that there exists $K_3>0$ such that for all $\ell>0$
\begin{align}
 \left\|\mathbbm{1}_{\{\chi=\ell\}}\right\|\leq K_3.\label{eq: chi = ell}
\end{align}
Additionally let $q\in(0,1)$, $R\in\left(0,1/q\right)$ and $\eta>1/q$ such that
for all $k\in\mathbb{N}$ we have that
$\mu\left(\chi=\eta^k\right)=R\cdot q^{k}$.

Further, let $\left(b_{n}\right)_{n\in\mathbb{N}}$ be a sequence
of natural numbers tending to infinity with $b_{n}=o\left(n\right)$.
If there exists $\psi\in\Psi$ such that 
\begin{align}
\lim_{n\rightarrow\infty}\frac{b_{n}}{\log\psi\left(\left\lfloor \log n\right\rfloor \right)}=\infty,\label{eq: t cond a4}
\end{align}
then there exists a positive valued sequence $\left(d_n\right)_{n\in\mathbb{N}}$ such that 
\begin{align*}
\lim_{n\rightarrow\infty}\frac{\mathsf{S}_n^{b_n}\chi}{d_n}=1\text{ a.s.}
\end{align*}
If additionally to \eqref{eq: t cond a4} there exists $\epsilon>0$ such that
$\left(b_n\right)$ can be written as $b_n=R/(1-q)\cdot q^{k_n}\cdot n+w_n$, 
where $w_n\geq \left(q^{k_n}\cdot n\right)^{1/2+\epsilon}\cdot \log \psi\left(\left\lfloor \log n\right\rfloor\right)^{1/2-\epsilon}$
and $w_n=o\left(q^{k_n}\right)$ and $\left(k_n\right)$ is a sequence of natural numbers,
then $d_n$ can be explicitly given by
\begin{align}
d_{n}=\frac{\eta}{q\cdot \eta-1}\cdot \left(R\cdot q\right)^{-\log \eta/\log q}\cdot \left(1-q\right)^{1+\log \eta/\log q}\cdot  n^{-\log\eta/\log q}\cdot b_n^{1+\log\eta/\log q}.\label{bn psi exp2}
\end{align}
\end{thm}

\subsection{General approach to the proof of Theorem \ref{thm: asymp thm}}\label{sec: proof main thm}
The proof of Theorem \ref{thm: asymp thm} is similar to the general structure given in \cite[Lemma 2.3]{kessebohmer_strong_2017}. 
However, due to the nature of having ties some additional attention is needed.

The first property considers the sum of truncated random variables.
Namely, for $\chi:\Omega\to\mathbb{R}_{\geq 0}$ 
and for a real valued sequence 
$\left(f_{n}\right)_{n\in\mathbb{N}}$ we let 
\begin{align*}
\mathsf{T}_n^{f_n}\chi&\coloneqq\upperleft{f_n}{}{\chi}+\upperleft{f_n}{}{\chi}\circ  T+\ldots +\upperleft{f_n}{}{\chi}\circ  T^{n-1}
\end{align*}
denote the corresponding truncated sum process.

Next, we will give some properties under which an intermediately trimmed strong law can be established.
\begin{Def}[{\cite[Definition 2.1]{kessebohmer_strong_2017}}]
 Let $\left(\Omega, \mathcal{A}, T , \mu,\mathcal{F},\left\|\cdot\right\|,\chi\right)$ fulfill Property $\mathfrak{D}$.
 We say that $(f_n)$ fulfills Property $\bm{A}$ for the system $\left(\Omega, \mathcal{A}, T , \mu,\mathcal{F},\left\|\cdot\right\|,\chi\right)$ if 
\begin{align*}
 \lim_{n\to\infty}\frac{\mathsf{T}_n^{f_n}\chi}{\int\mathsf{T}_n^{f_n}\chi\mathrm{d}\mu}=1\text{ a.s.}
\end{align*}
\end{Def}

The second property deals with the average number of large entries and is defined as follows:
\begin{Def}[{\cite[Definition 2.2]{kessebohmer_strong_2017}}]
 Let $\left(\Omega, \mathcal{A}, T , \mu,\mathcal{F},\left\|\cdot\right\|,\chi\right)$ fulfill Property $\mathfrak{D}$.
 We say that a tuple $((f_n),(\gamma_n))$ fulfills Property $\bm{B}$ for the system $\left(\Omega, \mathcal{A}, T , \mu,\mathcal{F},\left\|\cdot\right\|,\chi\right)$ if 
\begin{align*}
 \mu\left(\left|\sum_{i=1}^{n}\left(\mathbbm{1}_{\left\{\chi> f_n\right\}}\circ T^{i-1}-\mu\left(\chi>f_n\right)\right)\right|\geq \gamma_n\text{ i.o.}\right)=0.
\end{align*}
\end{Def}
A similar property will be given as follows:
\begin{Def}
 Let $\left(\Omega, \mathcal{A}, T , \mu,\mathcal{F},\left\|\cdot\right\|,\chi\right)$ fulfill Property $\mathfrak{D}$.
 We say that a tuple $((f_n),(\gamma_n'))$ fulfills Property $\bm{B'}$ for the system $\left(\Omega, \mathcal{A}, T , \mu,\mathcal{F},\left\|\cdot\right\|,\chi\right)$ if 
\begin{align*}
 \mu\left(\left|\sum_{i=1}^{n}\left(\mathbbm{1}_{\left\{\chi= f_n\right\}}\circ T^{i-1}-\mu\left(\chi=f_n\right)\right)\right|\geq \gamma_n'\text{ i.o.}\right)=0.
\end{align*}
\end{Def}

The following lemma is an extension of \cite[Lemma 2.3]{kessebohmer_strong_2017} giving an extra consideration 
to the appearance of ties.

\begin{lemma}\label{lem: Prop A B to trimming}
Let $\left(\Omega, \mathcal{A}, T , \mu,\mathcal{F},\left\|\cdot\right\|,\chi\right)$ fulfill Property $\mathfrak{D}$.
For the system $\left(\Omega, \mathcal{A}, T , \mu,\mathcal{F},\left\|\cdot\right\|,\chi\right)$ 
let further $(f_n)$ fulfill Property $\bm{A}$, let $((f_n),(\gamma_n))$ fulfill Property $\bm{B}$,
and let $((f_n),(\gamma_n'))$ fulfill Property $\bm{B'}$.
For fixed $w>0$ let 
\begin{align}
 r_n\in \left[\gamma_n,n\cdot \mu(\chi=f_n)+w(\gamma_n+\gamma'_n)\right],\label{eq: restr rn}
\end{align}
for all $n\in\mathbb{N}$.
If 
\begin{align}
 \limsup_{n\to\infty}\frac{\int \upperleft{f_n}{}{\chi}\mathrm{d}\mu}{\mu\left(\chi=f_n\right)\cdot f_n}>1\label{eq: cond chi=fn}
\end{align}
and
 \begin{align}
  \lim_{n\to\infty}\frac{\max\{\gamma_n, \gamma_n'\}\cdot f_n}{\int\mathsf{T}_n^{f_n}\chi\mathrm{d}\mu}=0\label{eq: fn gamman}
 \end{align}
hold, then  we have for $b_n\coloneqq \left\lceil n\cdot \mu\left(\chi>f_n\right)+r_n\right\rceil$ 
   that
\begin{align*}
 \lim_{n\to\infty}\frac{\mathsf{S}_n^{b_n}\chi}{\int\mathsf{T}_n^{f_n}\chi\mathrm{d}\mu-r_n\cdot f_n}=1\text{ a.s.}
\end{align*}
\end{lemma}
\begin{proof}
We can conclude from Property $\bm{B}$ that a.s.\ 
\begin{align*}
\mathsf{S}_{n}^{\left\lceil n\cdot \mu\left(\chi>f_n\right)+\gamma_n\right\rceil}\chi\leq \mathsf{T}_{n}^{f_{n}}\chi\text{ eventually.}
\end{align*}
Furthermore, Property $\bm{B}$ in conjunction with Property $\bm{B'}$ implies that a.s.\ eventually we have for all 
\begin{align*}
 k &\leq n\cdot \mu\left(\chi\geq f_n\right)-\gamma_n-\gamma_n'
\end{align*}
that $\chi\circ T^{\pi(k)}\geq f_n$.
This implies that a.s.\
\begin{align*}
 \mathsf{S}_{n}^{b_n}\chi\leq \mathsf{T}_{n}^{f_{n}}\chi-\min\left\{r_n, n\cdot \mu\left(\chi=f_n\right)-2\gamma_n-\gamma_n'\right\}\cdot f_n\;\;\;\text{ eventually.}
 \end{align*}
By the restriction of $r_n$ in \eqref{eq: restr rn} we have
\begin{align*}
 n\cdot \mu\left(\chi=f_n\right)-2\gamma_n-\gamma_n'\geq r_n-\left(2+w\right)\left(\gamma_n+\gamma_n'\right).
\end{align*}
This implies that we have a.s.\
\begin{align}
 \mathsf{S}_{n}^{b_n}\chi\leq \mathsf{T}_{n}^{f_{n}}\chi-\left(r_n-\left(2+w\right)\left(\gamma_n+\gamma_n'\right)\right)\cdot f_n\;\;\;\text{ eventually.}\label{San Tn1}
\end{align}

On the other hand, since ${^{f_n}}\chi\leq f_{n}$ it follows by
Property $\bm{B}$ that a.s.\ 
\begin{align}
\mathsf{T}_{n}^{f_{n}}\chi-\left(r_n+2\gamma_n\right)\cdot f_{n}\leq \mathsf{S}_{n}^{b_n}\chi\text{ eventually.}\label{Sbn >}
\end{align}

Combining \eqref{San Tn1} and \eqref{Sbn >} yields that a.s.\
\begin{align}
 \MoveEqLeft\frac{\mathsf{T}_{n}^{f_{n}}\chi-r_n\cdot f_n-2\gamma_n\cdot f_{n}}{\int\mathsf{T}_{n}^{f_{n}}\chi\mathrm{d}\mu-r_n\cdot f_n}\notag\\
 &\leq \frac{\mathsf{S}_{n}^{b_n}\chi}{\int\mathsf{T}_{n}^{f_{n}}\chi\mathrm{d}\mu-r_n\cdot f_n}\notag\\
 &\leq \frac{\mathsf{T}_{n}^{f_{n}}\chi-r_n\cdot f_n+\left(2+w\right)\cdot\left(\gamma_n+\gamma_n'\right)\cdot f_n}{\int\mathsf{T}_{n}^{f_{n}}\chi\mathrm{d}\mu-r_n\cdot f_n}\;\;\;\text{ eventually.}\label{eq: Snbn/Tntn-rn}
\end{align}
Using \eqref{eq: cond chi=fn} and \eqref{eq: fn gamman} together with the range of $r_n$ given in \eqref{eq: restr rn}
yields 
\begin{align*}
 \lim_{n\to\infty}\frac{2\gamma_n\cdot f_{n}}{\int\mathsf{T}_{n}^{f_{n}}\chi\mathrm{d}\mu-r_n\cdot f_n}
 =\lim_{n\to\infty}\frac{\left(2+w\right)\cdot\left(\gamma_n+\gamma_n'\right)\cdot f_n}{\int\mathsf{T}_{n}^{f_{n}}\chi\mathrm{d}\mu-r_n\cdot f_n}
 =0.
\end{align*}
Combining this with Property $\bm{A}$ and \eqref{eq: Snbn/Tntn-rn}
gives the statement of the lemma.
\end{proof}

In the next section we give a statement under which conditions Property $\bm{A}$ holds.

\subsection{Truncated random variables for St.\ Petersburg type distribution functions}\label{subsec: trunc sums}
We will first state a strong limit law for the truncated sum $\mathsf{T}_n^{f_n}\chi$.
\begin{thm}\label{thm: st peterburg}
Let $\left(\Omega, \mathcal{B},  T, \mu,\mathcal{F},\left\|\cdot\right\|,\chi\right)$ fulfill Property $\mathfrak{D}$.
For given $\eta>1$ 
assume $\mu\left(\chi=\eta^k\right)=R\cdot q^k$ with $q\in\left(1/\eta,1\right)$ and $R<1/q$, 
for all $k\in\mathbb{N}$. 
Let $\left(f_n\right)_{n\in\mathbb{N}}$ be a positive valued sequence with 
$F\left(f_n\right)>0$, for all $n\in\mathbb{N}$, and $\lim_{n\to\infty} f_n=\infty$.
If there exists $\psi\in\Psi$ such that
\begin{align}
 f_{n}^{-\log q/\log \eta}& =o\left(\frac{n}{\log\psi\left(\left\lfloor \log n\right\rfloor\right)}\right),\label{cond plateau}
\end{align}
then
\begin{align}
\lim_{n\to\infty}\frac{\mathsf{T}_{n}^{f_{n}}\chi}{\int\mathsf{T}_n^{f_n}\chi\mathrm{d}\mu}=1\text{ a.s.}\label{eq: limit peterburg}
\end{align}
\end{thm}

\begin{Rem}\label{Rem: St Peterburg}
 A similar setting was also studied in \cite{gyorfi_rate_2011} and \cite{nakata_limit_2015}.
 Particularly, \cite[Ex.\ 1.1]{nakata_limit_2015} considers the same distribution function for the i.i.d.\ setting with $q$ 
 restricted to $1/2$.
 A combination of Theorem 1.2 and Corollary 1.1 of \cite{nakata_limit_2015} gives the limit result as in \eqref{eq: limit peterburg}
 but imposes a stronger condition on $(f_n)$ than \eqref{cond plateau}.
\end{Rem}

Next we give some technical results which will help us to prove Theorem \ref{thm: st peterburg} and Theorem \ref{thm: asymp thm}.
\begin{lemma}
Assume that a sequence $(f_n)$ can be written as $f_n=\eta^{k_n}$
with $(k_n)$ a sequence of natural numbers. Then
\begin{align}
 \mu\left(\chi=f_n\right)&=R\cdot q^{\log f_n/\log \eta},\label{eq: mu chi =}\\
 \mu\left(\chi>f_n\right)&=\frac{R\cdot q^{\log f_n/\log \eta+1}}{1-q}.\label{eq: mu chi >}
\end{align} 
If additionally $f_n$ tends to infinity, then
\begin{align}
\int\upperleft{f_n}{}{\chi}\mathrm{d}\mu&\sim \frac{R\cdot q\cdot \eta}{q\cdot\eta-1}\cdot \left(q\cdot\eta\right)^{\log f_n/\log \eta}\label{eq: int chi sim}
\end{align}
\end{lemma}
\begin{proof}
As $f_n=\eta^{\log f_n/\log \eta}$ \eqref{eq: mu chi =} immediately follows.
This also gives
\begin{align*}
 \mu\left(\chi>f_n\right)
 =\sum_{k=\log_{\eta}f_n+1}^{\infty}\mu\left(\chi=\eta^k\right)
 =\sum_{k=\log_{\eta}f_n+1}^{\infty}R\cdot q^k
 =\frac{R\cdot q^{\log f_n/\log \eta +1}}{1-q},
\end{align*}
i.e.\ \eqref{eq: mu chi >} follows.
Finally,
\begin{align*}
\int\upperleft{f_n}{}{\chi}\mathrm{d}\mu
&=\sum_{k=1}^{\log_{\eta}f_n}\mu\left(\chi=\eta^k\right)\cdot\eta^k
=\sum_{k=1}^{\log_{\eta}f_n}R\cdot q^{k}\cdot\eta^k
=R\cdot\frac{\left(q\cdot\eta\right)^{\log_{\eta}f_n+1}-q\cdot\eta}{q\cdot\eta-1}\\
&\sim\frac{R\cdot q\cdot \eta}{q\cdot\eta-1}\cdot \left(q\cdot\eta\right)^{\log f_n/\log\eta}
\end{align*}
giving \eqref{eq: int chi sim}.
\end{proof}

In order to prove Theorem \ref{thm: st peterburg} we will make use of the following Lemma. 
\begin{lemma}[{\cite[Lemma 4.5]{kessebohmer_strong_2017}}]\label{lemma: Tnfn chi deviation}
Let $\left(\Omega, \mathcal{A},  T, \mu,\mathcal{F},\left\|\cdot\right\|,\chi\right)$ fulfill Property $\mathfrak{D}$.
Then there exist constants $N'\in\mathbb{N}$ and $E, K>0$ such that 
for all $\epsilon\in\left(0,E\right)$, $r>0$ with $F(r)>0$, and $n\in\mathbb{N}_{>N'}$ 
\begin{align*}
\mu\left(\max_{i\leq n}\left|\mathsf{T}^{r}_{i}\chi-\int\mathsf{T}_{i}^{r}\chi\mathrm{d}\mu\right|\geq \epsilon\cdot \int\mathsf{T}^{r}_{n}\chi\mathrm{d}\mu\right)
\leq K\cdot\exp\left(-\epsilon \cdot \frac{\int\mathsf{T}^{r}_{n}\chi\mathrm{d}\mu}{r}\right).
\end{align*}
\end{lemma}
(The lemma is stated slightly different in \cite{kessebohmer_strong_2017} but the statement 
in the current version becomes obvious from the proof of the lemma in \cite{kessebohmer_strong_2017}.)

\begin{proof}[Proof of Theorem \ref{thm: st peterburg}]
The proof is very similar to \cite[Proof of Theorem 2.5]{kessebohmer_strong_2017}, 
but as $F$ is not regularly varying the methods can not immediately be transferred. 

We define the sequences $\left(g_{n}\right)_{n\in\mathbb{N}}$ and $\left(\bar{g}_{n}\right)_{n\in\mathbb{N}}$
with $g_{n}\coloneqq\max\left\{ f_{n},n^{1/2}\right\} $
and $\bar{g}_{n}\coloneqq\min\left\{ f_{n},n^{1/2}\right\}$ and prove separately that 
\begin{align}
\mu\left(\left|\mathsf{T}_{n}^{g_{n}}\chi-\int\mathsf{T}_{n}^{g_{n}}\chi\mathrm{d}\mu\right|\geq\epsilon\int\mathsf{T}_{n}^{g_{n}}\chi\mathrm{d}\mu\text{ i.o.}\right)=0\label{eq: S g}
\end{align}
and 
\begin{align}
\mu\left(\left|\mathsf{T}_{n}^{\bar{g}_{n}}\chi-\int\mathsf{T}_{n}^{\bar{g}_{n}}\chi\mathrm{d}\mu\right|\geq\epsilon\int\mathsf{T}_{n}^{\bar{g}_{n}}\chi\mathrm{d}\mu\text{ i.o.}\right)=0\label{eq: S bar g}
\end{align}
hold.

To prove \eqref{eq: S g} we set $I_{j}\coloneqq\left[2^{j},2^{j+1}-1\right]$ for $j\in\mathbb{N}$
and obtain for every $m\in I_{j}$ and $r\in\R_{\geq 0}$ that 
\begin{align}
\left\{ \left|\mathsf{T}_{m}^{r}\chi-\int\mathsf{T}_{m}^{r}\chi\mathrm{d}\mu\right|\geq \epsilon\int\mathsf{T}_{m}^{r}\chi\mathrm{d}\mu\right\}
 & \subset\left\{ \left|\mathsf{T}_{m}^{r}\chi-\int\mathsf{T}_{m}^{r}\chi\mathrm{d}\mu\right|\geq \epsilon\int\mathsf{T}_{2^{j}}^{r}\chi\mathrm{d}\mu\right\} \notag\\
 & \subset\left\{ \left|\mathsf{T}_{m}^{r}\chi-\int\mathsf{T}_{m}^{r}\chi\mathrm{d}\mu\right|\geq\epsilon/2\int\mathsf{T}_{2^{j+1}-1}^{r}\chi\mathrm{d}\mu\right\} \notag\\
 & \subset\left\{ \max_{n\in I_{j}}\left|\mathsf{T}_{n}^{r}\chi-\int\mathsf{T}_{n}^{r}\chi\mathrm{d}\mu\right|\geq\epsilon/2\int\mathsf{T}_{2^{j+1}-1}^{r}\chi\mathrm{d}\mu\right\}.\label{max, nicht max}
\end{align}
Next we define the sequences $\left(s_{j}\right)_{j\in\mathbb{N}}$ and
$\left(t_{j}\right)_{j\in\mathbb{N}}$ as
\begin{align*}
s_{j}\coloneqq\left\lfloor \frac{j\cdot\log2}{2\cdot \log\eta}\right\rfloor 
\text{ and }
t_{j}\coloneqq{\left\lceil \frac{\log\left(\max_{n\in I_{j}}g_{n}\right)}{\log\eta}\right\rceil}.
\end{align*}
These numbers are chosen such that $\left[\eta^{s_j},\eta^{t_j}\right]\supset\left[\min_{n\in I_{j}}g_n,\max_{n\in I_{j}}g_n\right]$.
Furthermore, we know that $f_n\in [\eta^k,\eta^{k+1})$ implies $\mathsf{T}_n^{f_n}\chi=\mathsf{T}_n^{\eta^k}\chi$.
Hence, \eqref{max, nicht max} implies
\begin{align}
 \MoveEqLeft\bigcup_{n\in I_j}\left\{\left|\mathsf{T}_{n}^{g_{n}}\chi-\int\mathsf{T}_{n}^{g_{n}}\chi\mathrm{d}\mu\right|\geq\epsilon\int\mathsf{T}_{n}^{g_{n}}\chi\mathrm{d}\mu\right\}\notag\\
 &\subset\bigcup_{k=s_j}^{t_j}\left\{ \max_{n\in I_{j}}\left|\mathsf{T}_{n}^{\eta^k}\chi-\int\mathsf{T}_{n}^{\eta^k}\chi\mathrm{d}\mu\right|\geq\epsilon/2\int\mathsf{T}_{2^{j+1}-1}^{\eta^k}\chi\mathrm{d}\mu\right\},\label{eq: subset Ij sjtj}
\end{align}
for all $j\in\mathbb{N}$.
In order to estimate the sets on the right hand side of \eqref{eq: subset Ij sjtj} 
we can apply Lemma \ref{lemma: Tnfn chi deviation}
assuming that $\epsilon/2<E$
to the sum $\mathsf{T}_{n}^{\eta^{k}}\chi$ and obtain 
for $j$ sufficiently large 
\begin{align}
\mu\left( \max_{n\in I_{j}}\left|\mathsf{T}_{n}^{\eta^k}\chi-\int\mathsf{T}_{n}^{\eta^k}\chi\mathrm{d}\mu\right|\geq\epsilon/2\int\mathsf{T}_{2^{j+1}-1}^{\eta^k}\chi\mathrm{d}\mu\right)
\leq K\cdot\exp\left(-\epsilon/2\cdot\frac{\int\mathsf{T}_{2^{j+1}-1}^{\eta^{k}}\chi\mathrm{d}\mu}{\eta^{k}}\right).\label{eq: Tnrhok} 
\end{align}
Furthermore, \eqref{eq: int chi sim} implies
\begin{align}
\int\mathsf{T}_{2^{j+1}-1}^{\eta^{k}}\chi\mathrm{d}\mu
\asymp 2^j\cdot \eta^{k\cdot\left(1+\log q/\log\eta\right)}\asymp 2^j\cdot \eta^k\cdot q^k.\label{eq: asymp k j}
\end{align}
Here, we write $a_\ell\asymp b_\ell$ if there exists a constant $c>0$ such that $c^{-1}a_\ell\leq b_\ell\leq c a_\ell$ for all $\ell\in\mathbb{N}$
and in \eqref{eq: asymp k j} $\asymp$ holds both with respect to $k$ and $j$.
Hence, we obtain by \eqref{eq: Tnrhok} that there exists $W>0$ such that
\begin{align*}
\mu\left( \max_{n\in I_{j}}\left|\mathsf{T}_{n}^{\eta^k}\chi-\int\mathsf{T}_{n}^{\eta^k}\chi\mathrm{d}\mu\right|\geq\epsilon/2\int\mathsf{T}_{2^{j+1}-1}^{\eta^k}\chi\mathrm{d}\mu\right)
& \leq K\cdot \exp\left(-\epsilon/2\cdot W\cdot 2^j \cdot q^k\right),
\end{align*}
for all $k\in\mathbb{N}$ and $j$ sufficiently large.
This implies 
\begin{align}
\MoveEqLeft\sum_{k=s_{j}}^{t_{j}}\mu\left( \max_{n\in I_{j}}\left|\mathsf{T}_{n}^{\eta^k}\chi-\int\mathsf{T}_{n}^{\eta^k}\chi\mathrm{d}\mu\right|\geq\epsilon/2\int\mathsf{T}_{2^{j+1}-1}^{\eta^k}\chi\mathrm{d}\mu\right)\notag\\
 & \leq K\cdot\sum_{k=s_{j}}^{t_{j}}\exp\left(-\epsilon/2\cdot W\cdot 2^j \cdot q^k\right)\notag\\
 & \leq K\cdot\sum_{k=0}^{t_{j}}\exp\left(-\epsilon/2\cdot W\cdot 2^j \cdot q^k\right)\notag\\
 & = K\cdot\exp\left(-\epsilon/2\cdot W\cdot 2^j \cdot q^{t_j}\right)\cdot \sum_{k=0}^{t_{j}}\exp\left(-\epsilon/2\cdot W\cdot 2^j \cdot q^{t_j}\cdot\left(1/q^k-1\right)\right)\notag\\
 & \leq K\cdot\exp\left(-\epsilon/2\cdot W\cdot 2^j \cdot q^{t_j}\right)\cdot \sum_{k=0}^{\infty}\exp\left(-\epsilon/2\cdot W\cdot 2^j \cdot q^{t_j}\cdot\left(1/q-1\right)\cdot k\right)\notag\\  
 & = K\cdot \frac{\exp\left(-\epsilon/2\cdot W\cdot 2^j \cdot q^{t_j}\right)}{1-\exp\left(-\epsilon/2\cdot W\cdot 2^j \cdot q^{t_j}\cdot\left(1/q-1\right)\right)}.\label{reg var 1}
\end{align}
To continue we state the following technical lemma which is \cite[Lemma 5]{kessebohmer_strong_2016}.
\begin{lemma}\label{log gamma log tilde gamma}
Let $a,b>1$ and $\psi\in\Psi$. Then there exists $\omega\in\Psi$ such that 
\begin{align*}
\omega\left(\left\lfloor \log_b n\right\rfloor \right)\leq \psi\left(\left\lfloor \log_a n\right\rfloor\right).
\end{align*}
\end{lemma}

Noticing that 
\begin{align*}
2^j\cdot q^{t_j}
&\geq 2^j\cdot \max_{n\in I_j} g_n^{\frac{\log q}{\log \eta}}\cdot q
\geq \min_{n\in I_j}\left(n\cdot g_n^{\frac{\log q}{\log \eta}}\right)\cdot \frac{q}{2}
\geq \min_{n\in I_j}\left(n\cdot f_n^{\frac{\log q}{\log \eta}}\right)\cdot \frac{q}{2}
\end{align*}
and using condition \eqref{cond plateau} together with Lemma \ref{log gamma log tilde gamma}
implies that there exists $\psi\in \Psi$ such that 
\begin{align*}
 \lim_{j\to\infty}\frac{\epsilon/2\cdot W\cdot 2^j\cdot q^{t_j}}{\log\psi\left( j\right)}=\infty.
\end{align*}
Inserting this into the calculation in \eqref{reg var 1} yields
\begin{align*}
\sum_{k=s_{j}}^{t_{j}}\mu\left( \max_{n\in I_{j}}\left|\mathsf{T}_{n}^{\eta^k}\chi-\int\mathsf{T}_{n}^{\eta^k}\chi\mathrm{d}\mu\right|\geq\epsilon/2\int\mathsf{T}_{2^{j+1}-1}^{\eta^k}\chi\mathrm{d}\mu\right)
 &\leq K\cdot \frac{\exp\left(- \log\psi\left( j\right)\right)}{2}
 =K\cdot \frac{\psi(j)}{2}, 
\end{align*}
for $j$ sufficiently large.
Using \eqref{eq: subset Ij sjtj}, summing over the above quantity with respect to $j$, using the fact that $\psi\in \Psi$, and applying the Borel-Cantelli lemma yields \eqref{eq: S g}.

On the other hand using Lemma \ref{lemma: Tnfn chi deviation} again yields 
\begin{align*}
 \mu\left( \left|\mathsf{T}_{n}^{\bar{g}_n}\chi-\int\mathsf{T}_{n}^{\bar{g}_n}\chi\mathrm{d}\mu\right|\geq\epsilon\int\mathsf{T}_{n}^{\bar{g}_n}\chi\mathrm{d}\mu\right)
&\leq K\cdot\exp\left(-\epsilon\cdot\frac{\int\mathsf{T}_{n}^{\bar{g}_n}\chi\mathrm{d}\mu}{\bar{g}_n}\right)\\
&= K\cdot\exp\left(-\epsilon\cdot n\cdot \frac{\int\upperleft{\bar{g}_{n}}{}{\chi}\mathrm{d}\mu}{\bar{g}_{n}}\right)\\
&\leq K\cdot\exp\left(-\epsilon\cdot \frac{n}{\bar{g}_n}\right)
\leq K\cdot\exp\left(-\epsilon\cdot n^{1/2}\right),
\end{align*}
for $\epsilon\in\left(0,E\right)$ and $n$ sufficiently large.
Summing over the above quantity and using the Borel-Cantelli lemma yields that for all $\epsilon>0$ \eqref{eq: S bar g} holds
giving the statement of the theorem.
\end{proof}

\subsection{Proof of Theorem \ref{thm: asymp thm}}\label{subsec: proof asymp thm}
We start with two lemmas regarding Properties $\bm{B}$ and $\bm{B'}$. 
Before stating them we define
\begin{align}
c\left(k,n\right)\coloneqq c_{\epsilon,\psi}\left(k,n\right)\coloneqq\left(\max\left\{ k,\log\psi\left(\left\lfloor \log n\right\rfloor \right)\right\} \right)^{1/2+\epsilon}\cdot\left(\log\psi\left(\left\lfloor \log n\right\rfloor \right)\right)^{1/2-\epsilon},\label{c(n)}
\end{align}
for $k\in\mathbb{R}_{\geq1}$, $n\in\mathbb{N}_{\geq3}$,  $0<\epsilon<1/4$,
and $\psi\in\Psi$.

\begin{lemma}[{{\cite[Lemma 2.8]{kessebohmer_strong_2017}}}]\label{lem: bernoulli}
Let $\left(u_n\right)_{n\in\mathbb{N}}$ be a positive valued sequence
and define $p_n\coloneqq\mu\left(\chi> u_n\right)$.
Then there exist constants $N,V>0$ such that
for all $0<\epsilon<1/4$, $\psi\in\Psi$, $n\in\N_{>N}$, and $\left(u_n\right)$ positive valued we have
\begin{align*}
\mu\left(\left\{ \left|\mathsf{S}_n\mathbbm{1}_{\left\{\chi>u_n\right\}}-p_n\cdot n\right|\geq V\cdot c_{\epsilon,\psi}\left(p_{n}\cdot n,n\right)\text{ i.o.}\right\} \right)=0.
\end{align*}
\end{lemma}

\begin{lemma}\label{lem: bernoulli=}
Let $\left(v_n\right)_{n\in\mathbb{N}}$ be a positive valued sequence
and define $p_n'\coloneqq\mu\left(\chi= v_n\right)$.
Then there exist constants $N',V'>0$ such that
for all $0<\epsilon<1/4$, $\psi\in\Psi$, $n\in\N_{>N'}$, and $\left(v_n\right)$ positive valued we have
\begin{align*}
\mu\left(\left\{ \left|\mathsf{S}_n\mathbbm{1}_{\left\{\chi=v_n\right\}}-p_n'\cdot n\right|\geq V'\cdot c_{\epsilon,\psi}\left(p_n'\cdot n,n\right)\text{ i.o.}\right\} \right)=0.
\end{align*}
\end{lemma}

As the proof of Lemma \ref{lem: bernoulli=} is mainly the same as the proof of Lemma \ref{lem: bernoulli} given in \cite{kessebohmer_strong_2017},
we will not repeat it here.
The only difference is that we use \eqref{eq: chi = ell} instead of \eqref{C 2}.

With those two properties at hand we are able to prove Theorem \ref{thm: asymp thm}. 
\begin{proof}[Proof of Theorem \ref{thm: asymp thm}]
In the first part of the proof we will show \eqref{eq: lim Snbnchi}
using Lemma \ref{lem: Prop A B to trimming}.
We define 
\begin{align}
f_{n}\coloneqq F^{\leftarrow}\left(1-\frac{b_{n}-V\cdot c\left(b_{n},n\right)}{n}\right)\label{eq: def fn}
\end{align}
with $c$ as in \eqref{c(n)} and $V$ given in Lemma \ref{lem: bernoulli}.
With this choice of $(f_n)$ we have that 
\begin{align*}
 n\cdot\mu(\chi>f_n)\leq b_n-V\cdot c(b_n,n)\leq n\cdot \mu(\chi\geq f_n).
\end{align*}

We first want to show the following:
Let $(b_n)$ be as in \eqref{eq: t cond a4} and set
\begin{align}
\gamma_n\coloneqq V\cdot c\left(n\cdot\mu\left(\chi>f_n\right),n\right)\text{ and }
\gamma_n'\coloneqq V'\cdot c\left(n\cdot\mu\left(\chi=f_n\right),n\right),\label{eq: gamman gamman'}
\end{align}
for all $n\in\mathbb{N}$.
Then there exists $w\in\mathbb{R}_{>0}$ such that for all $n\in\mathbb{N}$ 
the sequence $(b_n)$ can be written as 
\begin{align}
 b_n=\left\lceil n\cdot\mu(\chi>f_n)+ r_n\right\rceil\label{eq: bn=}
\end{align}
with
\begin{align*}
 r_n\in \left[\gamma_n,n\cdot \mu(\chi=f_n)+w(\gamma_n+\gamma'_n)\right].
\end{align*}
This can be seen as follows: As a first boundary case set
$b_n'\coloneqq\left\lceil n\cdot\mu(\chi>f_n)+\gamma_n\right\rceil$,
then 
\begin{align}
 F^{\leftarrow}\left(1-\frac{b_{n}'-V\cdot c\left(b_{n}',n\right)}{n}\right)
 &\geq F^{\leftarrow}\left(1-\mu(\chi>f_n)\right)=f_n.\label{eq: F leftarrow bn'}
\end{align}
As the second boundary case set
\begin{align*}
 b_n''
 &\coloneqq\left\lceil n\cdot\mu(\chi>f_n)+ n\cdot \mu(\chi=f_n)+w(\gamma_n+\gamma'_n)\right\rceil
 =\left\lceil n\cdot\mu(\chi\geq f_n)+w(\gamma_n+\gamma'_n)\right\rceil.
\end{align*}
Assume that $\left(b_{n}\right)=\left(b_{n}''\right)$ fulfills \eqref{eq: t cond a4} 
for some $\psi\in\Psi$.
If we consider $c_{\epsilon,\psi}\left(k,n\right)$ 
in \eqref{c(n)} for the same $\psi$, then 
$c\left(b_{n}'',n\right)=o\left(b_{n}''\right)$.
Since $n\cdot \mu\left(\chi=f_n\right)\leq b_n''$ and
$c$ is monotonically increasing in its first argument, 
we also have that $w\cdot\left(\gamma_n+\gamma_n'\right)=o\left(b_n''\right)$. 
Applying \eqref{eq: mu chi =} and \eqref{eq: mu chi >} gives
\begin{align}
 b_n''\asymp n\cdot\mu\left(\chi> f_n\right)\asymp n\cdot\mu\left(\chi= f_n\right),\label{eq: bn'' asymp}
\end{align}
which implies that there exists $w>0$ such that $V\cdot c\left(b_n'',n\right)< w\cdot\left(\gamma_n+\gamma_n'\right)$, 
for all $n\in\mathbb{N}$.
Hence,
\begin{align}
 F^{\leftarrow}\left(1-\frac{b_{n}''-V\cdot c\left(b_{n}'',n\right)}{n}\right)
 &\leq  F^{\leftarrow}\left(1-\mu(\chi\geq f_n)\right)
 =F^{\leftarrow}\left(1-\mu(\chi> f_n/\eta)\right)
 =f_n/\eta<f_n.\label{eq: F leftarrow bn''}
\end{align}
Finally, \eqref{eq: F leftarrow bn'}
and \eqref{eq: F leftarrow bn''}
imply $b_n\in\left[b_n',b_n''\right]$ and
\eqref{eq: bn=} follows.

Furthermore, by Lemma \ref{lem: bernoulli} and Lemma \ref{lem: bernoulli=} 
we have that the pair $((f_n),(\gamma_n))$ fulfills Property $\bm{B}$
and $((f_n),(\gamma_n'))$ fulfills Property $\bm{B'}$
for $(\gamma_n)$ and $(\gamma_n')$ being defined as in 
\eqref{eq: gamman gamman'}.

As in \eqref{eq: bn'' asymp} we can show that $b_n\asymp n\cdot \mu\left(\chi=f_n\right)$. 
This observation combined with \eqref{eq: t cond a4} and \eqref{eq: mu chi =} implies
\begin{align*}
 \lim_{n\to\infty}\frac{n\cdot \mu\left(\chi=f_n\right)}{\log\psi\left(\left\lfloor\log n\right\rfloor\right)}
 =\lim_{n\to\infty}\frac{n\cdot R \cdot f_n^{\log q/\log \eta}}{\log\psi\left(\left\lfloor\log n\right\rfloor\right)}
 =\infty
\end{align*}
which is equivalent to \eqref{cond plateau}
and Theorem \ref{thm: st peterburg} states 
that under this condition Property $\bm{A}$ holds.

From \eqref{eq: mu chi =} and \eqref{eq: int chi sim} it immediately follows
that \eqref{eq: cond chi=fn} holds.

Finally, we will prove \eqref{eq: fn gamman}.
Since $\mu(\chi=f_n)\asymp \mu(\chi>f_n)$ by \eqref{eq: mu chi =} and \eqref{eq: mu chi >},
we have by the definition of $c$ that also $\gamma_n\asymp \gamma_n'$ holds.
Using first \eqref{eq: int chi sim} and then \eqref{eq: mu ai aj} gives
\begin{align*}
 \frac{\max\{\gamma_n,\gamma_n'\}\cdot f_n}{\int\mathsf{T}_n^{f_n}\chi\mathrm{d}\mu}
 \asymp \frac{\gamma_n\cdot f_n}{n\cdot f_n^{1+\log q/\log \eta}}
 =\frac{\gamma_n}{n\cdot f_n^{\log q/\log \eta}}
 \asymp\frac{\gamma_n}{n\cdot \mu(\chi=f_n)}.
\end{align*}
Since $b_n\asymp n\cdot\mu(\chi=f_n)$ and $b_n$ fulfills \eqref{eq: t cond a4}, we have that
$c(n\cdot \mu(\chi=f_n),n)=o(n\cdot \mu(\chi=f_n))$.
Thus, $\gamma_n=o\left(n\cdot \mu\left(\chi=f_n\right)\right)$
implying \eqref{eq: fn gamman}.

Hence, we can apply Lemma \ref{lem: Prop A B to trimming} and obtain the first part of the theorem. 
\medskip

Next we show the asymptotic given in \eqref{bn psi exp2}. The first step is to prove that 
$f_n=\eta^{k_n}$ for the definition of $f_n$ in \eqref{eq: def fn}. 
First note that 
$w_n\geq \left(q^{k_n}\cdot n\right)^{1/2+\epsilon}\cdot \log\psi\left(\left\lfloor\log n\right\rfloor\right)^{1/2-\epsilon}$
together with \eqref{eq: t cond a4} implies $w_n\geq c_{\epsilon/2,\psi}\left(R/\left(1-q\right)\cdot q^{k_n}\cdot n,n\right)$,
for $n$ sufficiently large.
As $w_n=o\left(q^{k_n}\cdot n\right)$ this also implies 
$w_n\geq V\cdot c_{\epsilon/2,\psi}\left(b_n,n\right)$,
for $n$ sufficiently large.
As we can choose $\epsilon$ arbitrarily for $c_{\epsilon,\psi}$ in \eqref{eq: def fn}, 
$f_n$ can be set such that
$n\cdot R/\left(1-q\right)\cdot q^{k_n} \geq n\cdot \mu\left(\chi>f_n\right)$. 
This together with \eqref{eq: mu chi >} implies $f_n\leq \eta^{k_n}$.

On the other hand, as $w_n=o\left(n\cdot q^{k_n}\right)$ we also have that 
\begin{align*}
 b_n < n\cdot \frac{R}{1-q}\cdot q^{k_n-1}
 = n\cdot \mu\left(\chi>\eta^{k_n-1}\right)
\end{align*}
implying $f_n>\eta^{k_n-1}$ and thus $f_n=\eta^{k_n}$.
In particular, it follows that $n\cdot \mu(\chi>f_n)\sim b_n$.

Using \eqref{eq: mu chi >} yields
\begin{align*}
 b_n^{1+\log\eta/\log q}
 \sim f_n^{\log q/\log \eta+1}\cdot \left(\frac{R\cdot q}{1-q}\right)^{1+\log\eta/\log q} \cdot n^{1+\log\eta/\log q}.
\end{align*}
Hence, using \eqref{eq: int chi sim} gives
\begin{align*}
 \int\mathsf{T}_n^{f_n}\mathrm{d}\mu
 &\sim \frac{\eta}{q\cdot \eta-1}\cdot \left(R\cdot q\right)^{-\log \eta/\log q}\cdot \left(1-q\right)^{1+\log \eta/\log q}\cdot  n^{-\log\eta/\log q}\cdot b_n^{1+\log\eta/\log q}.
\end{align*}
\end{proof}

\section{Properties of the space of quasi-H\"older continuous functions}\label{sec: quasi Hoelder}
\subsection{Further definitions and remarks}\label{subsec: def and rem}
We first give some standard definitions which we have omitted in the introduction. 
\begin{Def}\label{irred aperiod}
A matrix $A\in\left\{0,1\right\}^{k\times k}$ is called \emph{irreducible} if for each pair $\left(i,j\right)$ with $1\leq i,j\leq k$, there exists an $n\in\mathbb{N}$ such that $A^n\left(i,j\right)>0$. 
We define the \emph{period} $d$ of $A$ as 
\begin{align*}
d\coloneqq\gcd\left\{n\in\mathbb{N}\colon A^n\left(i,i\right)>0\text{, for all }1\leq i\leq k\right\}, 
\end{align*}
where $\gcd$ denotes the greatest common divisor. 
The matrix $A$ is called \emph{aperiodic} if $d=1$.
\end{Def}

\begin{Def}\label{def: gibbs measure}
  A probability measure $\mu$ on $\X $ is called \emph{Gibbs measure}
  if there exist a measurable function $f\colon \X \to \mathbb{R}$ and constants $K>0$ and $C\in \mathbb{R}$ such that
  for all $n\in\mathbb{N}$, all $\mathsf{A}\in\mathbb{A}_{n}$, and all $x\in [\mathsf{A}]$ we have that
  \begin{align*}
   K^{-1}
   \leq\frac{\mu\left([\mathsf{A}]\right)}{\exp\left(\mathsf{S}_n f(x) +n\cdot C\right)}
   \leq K.
  \end{align*}
\end{Def}

\begin{Def}\label{Def markov measure}
Let $\mathbb{A}=\left\{\mathfrak{a}_1,\ldots,\mathfrak{a}_k\right\}$, $k\geq 2$ be a finite alphabet.
Further, let $A\in \left\{0,1\right\}^{k\times k}$ and $P\in \left[0,1\right]^{k\times k}$ with $p_{i,j}= 0$ if and only if $a_{i,j}=0$ and $\sum_{j=1}^k p_{i,j}=1$, for all $1\leq i\leq n$. Further, let $\pi$ be an $n$-vector such that $\pi\cdot P=\pi$. 
If a measure $\mu$ is defined on cylinder sets as
\begin{align*}
\mu\left(\left[\mathfrak{a}_{i_1}\ldots \mathfrak{a}_{i_k}\right]\right)\coloneqq\pi_{i_1}\cdot p_{i_1,i_2}\cdot\ldots\cdot p_{i_{k-1},i_k}, 
\end{align*}
then it is called \emph{Markov measure} corresponding to $A$. 
\end{Def}

Next we recall the definition of $g$-functions and $g$-measures.
\begin{Def}\label{def: g func meas}
A measurable function $g\colon \X \to\mathbb{R}_{\geq 0}$ is called \emph{$g$-function} if
\begin{equation*}
 \sum_{x \in \sigma^{-1}y} g(x) = 1,
\end{equation*}
for all $y \in \X $.
If $g=\mathrm{d}\mu/\mathrm{d}\mu\circ\sigma$,
then $\mu$ is called \emph{$g$-measure} for a $g$-function $g$.
\end{Def}

For the following let $L_{f}\colon\mathcal{L}^1\to\mathcal{L}^1$ denote the Perron-Frobenius operator given by
\begin{align}
 {L}_f w\left(x\right)\coloneqq\sum_{\sigma y=x}\mathrm{e}^{f\left(y\right)}w\left(y\right),\quad x\in \X .\label{eq: PF-opearator}
\end{align}
If $\exp(f)$ is a $g$-function, then $L_{f}$ is normalized, i.e.\ $L_{f} 1=1$. 

\begin{Prop}\label{prop: g measure}
Let $f\colon \X \to \X $ be Lipschitz continuous with respect to the metric $d_1$ such that $\exp(f)$ is a $g$-function.  
Then there exists a  corresponding $g$-measure $\mu$. This measure is a $\sigma$-invariant  Gibbs measure with constant $C=0$ 
and as such mixing. Moreover, the Perron-Frobenius operator $L_f$ coincides with the definition of the transfer operator given in \eqref{hat CYRI}.
\end{Prop}
This proposition follows from a direct application of 
\cite[Theorem on p.\ 134]{keane_strongly_1972}, \cite[Corollary 3.2.1]{parry_zeta_1990} 
and \cite[Corollary 3.3]{walters_ruelle_1975}.
The last statement follows by direct calculation.

\subsection{The spectral gap property}\label{subsec: spec gap}
To show that the transfer operator $\widehat{\sigma}$ has a spectral gap on $\mathsf{H}$
we first have to ensure that $\mathsf{H}$ is a Banach space.
\begin{lemma}\label{lem: H Banach algebra}
$\mathsf{H}$ is a Banach algebra containing the constant functions and fulfills $\left|\cdot\right|_{\infty}\leq \left\|\cdot\right\|_{\epsilon_0}$, for each $\epsilon_0\leq 1$.
\end{lemma}
\begin{proof}
Let $f,g\in\mathsf{H}$.
Since 
\begin{align*}
\left\|f+g\right\|_{\epsilon_0}&=\left|f+g\right|_{\infty}+\sup_{0<\epsilon\leq\epsilon_0}\frac{\int\osc\left(f+g,B\left(\epsilon,x\right)\right)\mathrm{d}\mu\left(x\right)}{\epsilon}\\
&\leq \left|f\right|_{\infty}+\left|g\right|_{\infty}+\sup_{0<\epsilon\leq\epsilon_0}\frac{\int\osc\left(f,B\left(\epsilon,x\right)\right)\mathrm{d}\mu\left(x\right)+\int\osc\left(g,B\left(\epsilon,x\right)\right)\mathrm{d}\mu\left(x\right)}{\epsilon}\\
&\leq \left|f\right|_{\infty}+\left|g\right|_{\infty}+\sup_{0<\epsilon\leq\epsilon_0}\frac{\int\osc\left(f,B\left(\epsilon,x\right)\right)\mathrm{d}\mu\left(x\right)}{\epsilon}+\sup_{0<\epsilon\leq\epsilon_0}\frac{\int\osc\left(g,B\left(\epsilon,x\right)\right)\mathrm{d}\mu\left(x\right)}{\epsilon}\\
&=\left\|f\right\|_{\epsilon_0}+\left\|g\right\|_{\epsilon_0}
\end{align*}
and all other properties of a norm follow immediately, we find that $\left\|\cdot\right\|_{\epsilon_0}$ is a norm.

In the next steps we will show completeness following the proof in \cite[Lemma 2.3.17]{blank_discreteness_1997}. Let $\left(f_n\right)$ be a Cauchy sequence with respect to $\left\|\cdot\right\|_{\epsilon_0}$. 
In particular, $\left(f_n\right)$ is also a Cauchy sequence with respect to $\left|\cdot\right|_{\infty}$, we set $f$ as its limit. 
So our next step is to prove that $f\in \mathsf{H}$. Since $\left(f_n\right)$ is a Cauchy sequence with respect to $\left\|\cdot\right\|_{\epsilon_0}$, for each $\delta>0$ we can choose $L>0$
such that $\left\|f_{k}-f_{\ell}\right\|_{\epsilon_0}<\delta$ for all $k,\ell>L$. Then we have that
\begin{align*}
 \left\|f_{k}-f_{\ell}\right\|_{\epsilon_0}
 =\left|f_{k}-f_{\ell}\right|_{\infty}+\sup_{0<\epsilon\leq\epsilon_0}\frac{\int\osc\left(f_{k}-f_{\ell},B\left(\epsilon,x\right)\right)\mathrm{d}\mu\left(x\right)}{\epsilon}<\delta.
\end{align*}
By Fatou's lemma we have that the limit $\ell\to \infty$ on the right hand side exists and thus,
\begin{align*}
 \left\|f_{k}-f\right\|_{\epsilon_0}=\left|f_{k}-f\right|_{\infty}+\sup_{0<\epsilon\leq\epsilon_0}\frac{\int\osc\left(f_{k}-f,B\left(\epsilon,x\right)\right)\mathrm{d}\mu\left(x\right)}{\epsilon}<\delta.
\end{align*}
Thus, $f\in\mathsf{H}$ and $\left(f_{n}\right)$ converges to $f$ with respect to $\left\|\cdot\right\|_{\epsilon_0}$. 

We further note that for all $f,g\in \mathsf{H}$
\begin{align*}
\left|f\cdot g\right|_{\epsilon_0}\leq\left|f\right|_{\infty}\cdot\left|g\right|_{\epsilon_0}+\left|f\right|_{\epsilon_0}\cdot\left|g\right|_{\infty} 
\end{align*}
and thus,
\begin{align*}
\left\|f\cdot g\right\|_{\epsilon_0}
&\leq \left|f\right|_{\infty}\cdot\left|g\right|_{\infty}+\left|f\right|_{\infty}\cdot\left|g\right|_{\epsilon_0}+\left|f\right|_{\epsilon_0}\cdot\left|g\right|_{\infty} 
\leq \left\|f\right\|_{\epsilon_0}\cdot\left\|g\right\|_{\epsilon_0}.
\end{align*}
Hence, $\left(\mathsf{H},\left\|\cdot\right\|_{\epsilon_0}\right)$ is a Banach algebra, the constant functions are obviously contained, and $\left|\cdot\right|_{\infty}\leq \left\|\cdot\right\|_{\epsilon_0}$ is clear from the definition of $\left\|\cdot\right\|_{\epsilon_0}$.  
\end{proof}

To prove that $\widehat{\sigma}$ has a spectral gap, we use the following theorem by Hennion and Hervé which is a generalization of a theorem by Doeblin and Fortet, \cite{doeblin_sur_1937}, and Ionescu–Tulcea and Marinescu, \cite{ionescu_theorie_1950}.
\begin{lemma}[{\cite[Theorem II.5]{hennion_limit_2001}}]\label{h it m d f}
Suppose $\left(\mathcal{L} ,\left\|\cdot\right\|\right)$ is a Banach space and $U:\mathcal{L}\to\mathcal{L}$ is a bounded linear
operator with spectral radius $1$. Assume that there exists a semi-norm $\left|\cdot\right|'$ with
the following properties:
\begin{enumerate}[(a)]
\item\label{Hennion 1} $\left|\cdot\right|'$ is continuous on $\mathcal{L}$.
\item\label{Hennion 2} $U$ is bounded on $\mathcal{L}$ with respect to $\left|\cdot\right|'$, i.e.\ there exists $M > 0$ such that $\left|U f \right|'\leq  M \left|f\right|'$, for all $f \in\mathcal{L}$.
\item\label{Hennion 3} There exist constants $0<r<1$, $R>0$, and $n_0\in\mathbb{N}$ such that 
\begin{align}
\left\|U^{n_0} f\right\|\leq r^{n_0}\cdot\left\|f\right\|+R\cdot \left|f\right|',\label{mar tul ios}
\end{align}
for all $f \in\mathcal{L}$.
\item\label{Hennion 4} 
$U\left\{f\in\mathcal{L}\colon\left|f\right|'<1\right\}$ is precompact on $\left(\mathcal{L},\left\|\cdot\right\|\right)$, i.e.\ for each sequence $\left(f_n\right)_{n\in\mathbb{N}}$ with values in $\mathcal{L}$ fulfilling $\sup_{n\in\mathbb{N}} \left|f_n\right|'<1$ there exists
a subsequence $\left(n_k\right)$ and $g \in \mathcal{L}$ such that  
\begin{align*}
\lim_{k\to\infty}\left\|U f_{n_k}-g\right\|=0.
\end{align*}
\end{enumerate}

Then $U$ is quasi-compact, i.e.\ there is a direct sum decomposition $\mathcal{L}=F\oplus H$ and $0<\tau<\rho\left(U\right)$ where
\begin{itemize}
\item $F$, $H$ are closed and $U$-invariant, i.e.\ $U\left(F\right) \subset F$, $U\left(H\right) \subset H$,
\item $\dim\left(F\right) <\infty$  and all eigenvalues of $U\lvert_F:\mathcal{F}\to \mathcal{F}$ have modulus larger than $\tau$, and
\item $\rho\left(U\lvert_{H}\right)<\tau$.
\end{itemize}
\end{lemma}

With the following lemma we will show that $\widehat{\sigma}$ 
has a spectral gap. 
It is a standard approach, but for completeness we will also give a proof of this lemma.
\begin{lemma}\label{quasi comp to spec gap}
If $U$ is quasi-compact, $U$ has a unique eigenvalue on $\left\{z \colon \left|z\right| = \rho\left(U\right)\right\}$, and this eigenvalue is simple, then $U$ has a spectral gap. 
\end{lemma}
\begin{proof}
We use the decomposition of $\mathcal{L}$ from Lemma \ref{h it m d f}. 
Since $F$ is finite dimensional, we can calculate the Jordan form of $L\lvert_{F}$. Since all eigenvalues of $L\lvert_{F}$ have modulus larger than $\tau$, the unique and simple eigenvalue on $\left\{z \colon \left|z\right| = \rho\left(U\right)\right\}$ is also unique and simple in $L\lvert_{F}$. Hence, the Jordan form consists of a $1\times 1$- block with eigenvalue $\lambda$ such that $\left|\lambda\right| = \rho\left(U\right)$, and possibly other Jordan blocks with eigenvalues $\lambda_i$ such that $\left|\lambda_i\right| < \left|\lambda\right|$ for each $i$.
Hence, $F$ can be decomposed into $F=\spann\left\{\nu\right\}\oplus F'$, where $U\nu=\lambda\nu$ and  $\rho\left(U\lvert_{F'}\right) < \left|\lambda\right|$. Thus, 
\begin{align*}
\mathcal{L}=\spann\left\{\nu\right\}\oplus F'\oplus H= \spann\left\{\nu\right\}\oplus H'
\end{align*}
where $U\left(H'\right)\subset H'$ and $\rho\left(U\lvert_{H'}\right) < \left|\lambda\right|$.

For the following we define $\Pi_1$ and $\Pi_2$ as the unique projections to the spaces $\spann\left\{\nu\right\}$ and $H'$, i.e.\ for every $f\in\mathcal{L}$ it holds $f=\Pi_1f+\Pi_2f$. These projections are idempotent, i.e.\ $\Pi_1^2=\Pi_1$ and $\Pi_2^2=\Pi_2$. 
Since $U\left(\spann\left\{\nu\right\}\right)\subset\spann\left\{\nu\right\}$ and $U\left(H'\right)\subset H'$, we have that 
\begin{align*}
P\coloneqq U\Pi_1=\Pi_1 U
\end{align*}
and 
\begin{align*}
N\coloneqq U\Pi_2=\Pi_2 U. 
\end{align*}
Since $\dim\left(\spann\left\{\nu\right\}\right)=1$, we have that $\dim\left(\Image\left(P\right)\right)=1$ and $P$ is obviously a projection.
Furthermore, we have 
\begin{align*}
PN=U\Pi_1U\Pi_2&=U\Pi_1U\left(\id-\Pi_1\right)=U^2\Pi_1\left(\id-\Pi_1\right)=U^2\left(\Pi_1-\Pi_1^2\right)=0 
\end{align*}
and analogously $NP=0$, which finishes the proof the lemma. 
\end{proof}

\begin{lemma}\label{spec gap}
Let $\widehat{ \sigma}$ be defined as in \eqref{hat CYRI}, then $\widehat{ \sigma}$ has a simple eigenvalue $\lambda=\rho\left(\widehat{\sigma}\right)=1$. This eigenvalue is unique on the unit circle and has maximal modulus. 
\end{lemma}
The proof is standard, but see also \cite[Lemma 3.2]{kessebohmer_strong_2017}.
For the next lemma let $L_f$ denote the Perron-Frobenius operator given in \eqref{eq: PF-opearator}. 
\begin{Prop}\label{prop: spec gap}
For $f\in F_{\theta}$, there exists $\epsilon_0\in\left(0,1\right)$ such that
$L_f$ is a bounded linear operator on $\mathsf{H}$ 
with respect to $\left\|\cdot\right\|_{\epsilon_0}$ and has a spectral gap.
\end{Prop}
\begin{proof}
We aim to apply Lemma \ref{h it m d f} in combination with Lemma \ref{quasi comp to spec gap} and Lemma \ref{spec gap} to the space $\left(\mathsf{H},\left\|\cdot\right\|_{\epsilon_0}\right)$ with
$\left|\cdot\right|_{1}$ as the semi-norm $\left|\cdot\right|'$. 
So we only have to show that $\left|\cdot\right|_{1}$ fulfills \eqref{Hennion 1} to \eqref{Hennion 4}.

{\em ad \eqref{Hennion 1}}:
Obviously, $\left|\cdot\right|_{1}$ is continuous. 

{\em ad \eqref{Hennion 2}}: 
Since $\mathrm{e}^f$ is positive and $L_f$ is normalized and the integral is $L_f$-invariant, it follows that 
\begin{align*}
\left|L_f w\right|_{1}
&=\int\left|\sum_{\sigma y=x}\mathrm{e}^{f\left(y\right)}w\left(y\right)\right|\mathrm{d}\mu\left(x\right)
\leq\int\sum_{\sigma y=x}\mathrm{e}^{f\left(y\right)}\left|w\left(y\right)\right|\mathrm{d}\mu\left(x\right)\\
&=\int L_f\left|w\right|\mathrm{d}\mu=\int\left|w\right|\mathrm{d}\mu=\left|w\right|_1,
\end{align*}
i.e.\  $L_f$ is bounded on $\mathsf{H}$ with respect to $\left|\cdot\right|_{1}$. 

{\em ad \eqref{Hennion 3}}:
Before we can start with the proof of \eqref{Hennion 3} we need the following two lemmas:
\begin{lemma}\label{lem: Gibbs gamma}
 There exists $\gamma\in\left(0,1\right)$, $K>0$ and $u\in\mathbb{N}$ such that for all $n\in\mathbb{N}$ and all $\mathsf{A}\in\mathbb{A}_n$
\begin{align}
 \mu\left(\left[\mathsf{A}\right]\right)\leq K\cdot \gamma^{\lfloor n/u\rfloor}.\label{eq: gibbs leq}
\end{align}
\end{lemma}
\begin{proof}
 By the Gibbs property, see Proposition \ref{prop: g measure},
there exists $K>0$ such that for all $n\in\mathbb{N}$, $\mathsf{A}\in \mathbb{A}_n$, and $x\in\left[\mathsf{A}\right]$ we have that 
\begin{align*}
 K^{-1}\leq \frac{\mu\left(\left[\mathsf{A}\right]\right)}{\exp\left(\mathsf{S}_n f\left(x\right)\right)}\leq K.
\end{align*}
Hence, 
\begin{align}
  \mu\left(\left[\mathsf{A}\right]\right)\geq K^{-1}\cdot \min_{z\in \X }\exp(f(z))^n.\label{eq: gibbs ineq}
 \end{align}
To prove \eqref{eq: gibbs leq} we use that $A$ is irreducible and aperiodic 
which implies that there is a maximal number $u$ such that all $x,\sigma(x),\ldots,\sigma^{u-2}(x)$ 
have only one preimage,
i.e.\ in this case $\sigma^{-1}(\sigma^{u-1}(x))$ consists of at least two points.

On the other hand, if $x$ has more than one preimage, then $\exp(f(y))<1$ for all $y$ fulfilling $\sigma(y)=x$. 
This implies 
\begin{align*}
 \exp(\mathsf{S}_nf(x))
 \leq \left(\sup_{\{z\in \X \colon \exp(f(z))<1\}}\exp(f(z))\right)^{\lfloor n/u\rfloor}.
\end{align*}
As $\exp(f)$ is bounded from below and $\sum_{\sigma(y)=x}\exp(f(y))=1$
the above term has to be less than one. 
Applying \eqref{eq: gibbs ineq} and setting 
$\gamma=\sup_{\{z\in \X \colon \exp(f(z))<1\}}\exp(f(z))$ 
gives the statement of \eqref{eq: gibbs leq}.
\end{proof}

\begin{lemma}\label{lem: existence of s}
 There exist $\epsilon_0,s\in(0,1)$ and $L\in\mathbb{N}$ such that for all $\ell\geq L$, 
 $\mathsf{A}\in\mathbb{A}_\ell$, $x\in \X $ with $\mathsf{A}x$ admissible, 
 and $\epsilon\in (0,\epsilon_0]$
 we have that $[\mathsf{A}]\cap \sigma^{-\ell}( B(\epsilon,x))\subset B(s\epsilon,\mathsf{A}x)$.
\end{lemma}
\begin{proof}
 For all $\epsilon>0$ and $x\in \X $ 
 there exists $n\in\mathbb{N}$ 
 such that $B(\epsilon,x)=[x_1\ldots x_n]$.
 Hence, if we write $[\mathsf{A}]=\mathsf{a}_{1}\ldots\mathsf{a}_{\ell}$, 
 then $[\mathsf{A}]\cap \sigma^{-\ell}( B(\epsilon,x))=[\mathsf{a}_{1}\ldots\mathsf{a}_{\ell}x_1\ldots x_n]$.
 By the Gibbs property from Proposition \ref{prop: g measure} we have 
\begin{align}
 \mu\left(\left[\mathsf{a}_{1}\ldots\mathsf{a}_{\ell}x_{1}\ldots x_{n}\right]\right)
 &\leq K\cdot \exp\left(\mathsf{S}_{\ell+n}f(\mathsf{A}x)\right)\notag\\
 &= K\cdot \exp\left(\mathsf{S}_{\ell}f(\mathsf{A}x)\right)\cdot \exp\left(\mathsf{S}_{n}f(x)\right)\notag\\
 &\leq K^3\cdot \mu\left(\left[\mathsf{a}_{1}\ldots\mathsf{a}_{\ell}\right]\right)\cdot \mu\left(\left[x_1\ldots x_n\right]\right)\notag\\
 &\leq K^3\cdot \mu\left(\left[\mathsf{a}_{1}\ldots\mathsf{a}_{\ell}\right]\right)\cdot \epsilon.\label{eq: mu ai aj}
\end{align}
Furthermore, 
 Lemma \ref{lem: Gibbs gamma} implies the existence of $L\in\mathbb{N}$ and $s\in (0,1)$ such that 
we have for all $\ell\geq L$ and all $\mathsf{A}\in\mathbb{A}_{\ell}$ that
$K^3\cdot \mu\left(\left[\mathsf{a}_{1}\ldots\mathsf{a}_{\ell}\right]\right)\leq s<1$.
Combining this consideration with \eqref{eq: mu ai aj} yields the statement of the lemma.
\end{proof}
Now we are in the position to begin with the proof of \eqref{Hennion 3}.
We have that
\begin{align*}
 \osc\left(L_f^{\ell}w,B\left(\epsilon,x\right)\right)
 &=\esssup_{y\in B\left(\epsilon, x\right)}\left(\sum_{\sigma^{\ell}\left(z\right)=y}\mathrm{e}^{\mathsf{S}_n\left(z\right)}\cdot w\left(z\right)\right)
 -\essinf_{y\in B\left(\epsilon, x\right)}\left(\sum_{\sigma^{\ell}\left(z\right)=y}\mathrm{e}^{\mathsf{S}_n\left(z\right)}\cdot w\left(z\right)\right).
\end{align*}
Since $\epsilon_0<1$ we have that $B\left(\epsilon,x\right)$ is a cylinder set of length at least one 
and we can conclude that 
\begin{align}
 \osc\left(L_f^{\ell}w,B\left(\epsilon,x\right)\right)
 &\leq \sum_{\substack{\mathsf{A}\in\mathbb{A}_{\ell},\\ \mathsf{A}x \text{ admissible}}}
 \esssup_{y\in B\left(\epsilon, x\right)}\left(\mathrm{e}^{\mathsf{S}_n\left(\mathsf{A}y\right)}\cdot w\left(\mathsf{A}y\right)\right)
 -\essinf_{y\in B\left(\epsilon, x\right)}\left(\mathrm{e}^{\mathsf{S}_n\left(\mathsf{A}y\right)}\cdot w\left(\mathsf{A}y\right)\right)\notag\\
 &=\sum_{\substack{\mathsf{A}\in\mathbb{A}_{\ell},\\ \mathsf{A}x \text{ admissible}}}
 \esssup_{z\in \sigma^{-\ell}\left(B\left(\epsilon, x\right)\right)\cap\left[\mathsf{A}\right]}\left(\mathrm{e}^{\mathsf{S}_n\left(z\right)}\cdot w\left(z\right)\right)
 -\essinf_{z\in \sigma^{-\ell}\left(B\left(\epsilon, x\right)\right)\cap\left[\mathsf{A}\right]}\left(\mathrm{e}^{\mathsf{S}_n\left(z\right)}\cdot w\left(z\right)\right)\notag\\
 &=\sum_{\substack{\mathsf{A}\in\mathbb{A}_{\ell},\\ \mathsf{A}x \text{ admissible}}}
 \osc\left(\mathrm{e}^{\mathsf{S}_n}\cdot w,\sigma^{-\ell}\left(B\left(\epsilon, x\right)\right)\cap\left[\mathsf{A}\right]\right).\label{eq: osc Lfl}
\end{align}
For the following calculations we assume that for given $\mathsf{A}\in\mathbb{A}_{\ell}$ and $x\in \X $ we have that $\mathsf{A}x$ is admissible and thus
$\left[\mathsf{A}\right]\cap\sigma^{-\ell}\left(B\left(\epsilon,x\right)\right)\neq \emptyset$.
Using \cite[Proposition 3.2 (iii)]{saussol_absolutely_2000} which can also be applied to the situation here 
yields
\begin{align}
\osc\left(\mathrm{e}^{\mathsf{S}_{\ell}f}\cdot w, \left[\mathsf{A}\right]\cap\sigma^{-\ell}\left(B\left(\epsilon,x\right)\right)\right)
&\leq \esssup_{z\in \left[\mathsf{A}\right]\cap\sigma^{-\ell}\left(B\left(\epsilon,x\right)\right)}\mathrm{e}^{\mathsf{S}_{\ell}f(z)}\cdot \osc\left(w, \left[\mathsf{A}\right]\cap\sigma^{-\ell}\left(B\left(\epsilon,x\right)\right)\right)\notag\\
&\qquad+\osc\left(\mathrm{e}^{\mathsf{S}_{\ell}f}, \left[\mathsf{A}\right]\cap\sigma^{-\ell}\left(B\left(\epsilon,x\right)\right)\right)\cdot \essinf_{z\in \left[\mathsf{A}\right]\cap\sigma^{-\ell}\left(B\left(\epsilon,x\right)\right)}\left|w\left(z\right)\right|.\label{eq: osc Lfw}
\end{align}

In order to estimate the first summand we notice that by Lemma \ref{lem: existence of s} there exists $s< 1$ 
such that for all $\mathsf{A}\in\mathbb{A}_\ell$ we have that
\begin{align}
\osc\left( w, \left[\mathsf{A}\right]\cap\sigma^{-\ell}\left(B\left(\epsilon,x\right)\right)\right)
&\leq \osc\left( w, B\left(s\cdot\epsilon,\mathsf{A}x\right)\right).\label{eq: osc Lfw 1a}
\end{align}
Furthermore, \eqref{eq: gibbs leq} implies that 
for all $\epsilon\in\left(0,1\right)$ there exists $n_{\epsilon}\in\mathbb{N}$ such that for all $x\in \X $
the set
$B\left(\epsilon,x\right)$ is a cylinder of length at least $n_{\epsilon}$.
Then $\mathsf{A}\in\mathbb{A}_{\ell}$ implies that 
$\left[\mathsf{A}\right]\cap\sigma^{-\ell}\left(B\left(\epsilon,x\right)\right)$ is a cylinder set of length at least $n+\ell$.
Let $R$ be the Lipschitz constant of $f$. 
Then we have that
\begin{align}
 \esssup_{z\in \left[\mathsf{A}\right]\cap\sigma^{-\ell}\left(B\left(\epsilon,x\right)\right)}\mathsf{S}_{\ell}f(z)
 \leq \mathsf{S}_{\ell}f(\mathsf{A}x)+R\cdot \sum_{j=0}^{\ell-1}\theta^{n_{\epsilon}+\ell-j-1}
 \leq \mathsf{S}_{\ell}f(\mathsf{A}x)+R\cdot \frac{\theta^{n_{\epsilon}}}{1-\theta}.\label{eq: osc Lfw 1b}
\end{align}

Noting that $\mathrm{e}^{\mathsf{S}_{\ell}f(x)}>0$ for all $x\in \X $
yields for the second summand of \eqref{eq: osc Lfw} that
\begin{align}
 \MoveEqLeft\osc\left(\mathrm{e}^{\mathsf{S}_{\ell}f}, \left[\mathsf{A}\right]\cap\sigma^{-\ell}\left(B\left(\epsilon,x\right)\right)\right)\cdot \essinf_{y\in \left[\mathsf{A}\right]\cap\sigma^{-\ell}\left(B\left(\epsilon,x\right)\right)}\left|w\left(y\right)\right|\notag\\
 &\leq \esssup_{z\in \left[\mathsf{A}\right]\cap\sigma^{-\ell}\left(B\left(\epsilon,x\right)\right)}\mathrm{e}^{\mathsf{S}_{\ell}f(z)}\cdot \left|w\left(\mathsf{A}x\right)\right|\notag\\
 &\leq \mathrm{e}^{\mathsf{S}_{\ell}f(\mathsf{A}x)}\cdot \left|w\left(\mathsf{A}x\right)\right|\cdot \exp\left(R\cdot \frac{\theta^{n_{\epsilon}}}{1-\theta}\right),\label{eq: osc Lfw 2}
\end{align}
where the last line follows from \eqref{eq: osc Lfw 1b}.

Combining \eqref{eq: osc Lfl}, \eqref{eq: osc Lfw}, \eqref{eq: osc Lfw 1a}, \eqref{eq: osc Lfw 1b}, and \eqref{eq: osc Lfw 2} yields 
\begin{align*}
 \osc\left(L_f^{\ell}w,B\left(\epsilon,x\right)\right)
 &\leq \sum_{\substack{\mathsf{A}\in\mathbb{A}_\ell\\ \mathsf{A}x\text{ admissible}}}\left(\mathrm{e}^{\mathsf{S}_{\ell}f(\mathsf{A}x)}\cdot\osc\left(w, \left[\mathsf{A}\right]\cap B\left(s\cdot\epsilon,\mathsf{A}x\right)\right)\cdot \exp\left(R\cdot \frac{\theta^{n_{\epsilon}}}{1-\theta}\right)\right.\\
&\qquad\left.+\mathrm{e}^{\mathsf{S}_{\ell}f(\mathsf{A}x)}\cdot \left|w\left(\mathsf{A}x\right)\right|\cdot \exp\left(R\cdot \frac{\theta^{n_{\epsilon}}}{1-\theta}\right)\right)\\
&=L_f^{\ell}\left(\osc\left(w, \left[\mathsf{A}\right]\cap B\left(s\cdot\epsilon,x\right)\right)+\left|w\right|\right)\cdot \exp\left(R\cdot \frac{\theta^{n_{\epsilon}}}{1-\theta}\right).
\end{align*}
Integrating with respect to $\mu$ and noting that the integral is $L_f$-invariant yields
\begin{align*}
 \int \osc\left(L_f^{\ell}w,B\left(\epsilon,x\right)\right)\mathrm{d}\mu
 \leq \int \left(\osc\left(w, \left[\mathsf{A}\right]\cap B\left(s\cdot\epsilon,\mathsf{A}x\right)\right)+\left|w\right|\right)\cdot \exp\left(R\cdot \frac{\theta^{n_{\epsilon}}}{1-\theta}\right)\mathrm{d}\mu.
\end{align*}
Using the definition of $\left|w\right|_{\epsilon_0}$ and assuming that $\epsilon<\epsilon_0$ yields
\begin{align*}
 \int \osc\left(L_f^{\ell}w,B\left(\epsilon,x\right)\right)\mathrm{d}\mu
 \leq \left|w\right|_{\epsilon_0}\cdot s\cdot \epsilon\cdot \exp\left(R\cdot \frac{\theta^{n_{\epsilon}}}{1-\theta}\right)
+\left|w\right|_1\cdot \exp\left(R\cdot \frac{\theta^{n_{\epsilon}}}{1-\theta}\right).
\end{align*}
If $n_{\epsilon}$ is large enough, we have that 
\begin{align}
 s\cdot \exp\left(R\cdot \frac{\theta^{n_{\epsilon}}}{1-\theta}\right)<1.\label{eq: cond nepsilon}
\end{align}
We choose $\epsilon_0$ sufficiently small such that \eqref{eq: cond nepsilon} 
is fulfilled for $\epsilon\coloneqq\epsilon_0$ and set $n_0\coloneqq n_{\epsilon_0}$ and 
$\xi\coloneqq s\cdot \exp\left(R\cdot \theta^{n_{0}}/\left(1-\theta\right)\right)$.
Using again the definition of $\left|w\right|_{\epsilon_0}$
gives 
\begin{align*}
 \left|L_f^{\ell}w\right|_{\epsilon_0}
 \leq \left|w\right|_{\epsilon_0}\cdot \xi 
+\left|w\right|_1\cdot \exp\left(C\cdot \frac{\theta^{n_0}}{1-\theta}\right)
\end{align*}
and thus the Hennion-Hervé inequality \eqref{mar tul ios} follows.

{\em ad \eqref{Hennion 4}}:
We first prove that $\mathsf{K}\coloneqq\{h:\left\|h\right\|_{\epsilon_0}<1\}$ is compact
using the approach of \cite[Lemma 2.3.18]{blank_discreteness_1997}.
From \eqref{eq: gibbs leq} we can conclude that for all $\epsilon>0$
there exists $M\in\mathbb{N}$ such that for all $m\geq M$ and $\mathsf{A}\in\mathbb{A}_m$ we have
$\mu\left(\left[\mathsf{A}\right]\right)\leq \epsilon$.

We define $\mathcal{B}_m$ as the sigma algebra generated by the cylinder sets of length $m$.
We note that the conditional expectations $\mathbb{E}\left(h\vert\mathcal{B}_m\right)$ 
are in particular piecewise constant functions.
For $\epsilon\leq \epsilon_0$ we have
\begin{align}
 \left|h-\mathbb{E}\left(h\vert\mathcal{B}_m\right)\right|_{1}
 &\leq \sum_{\mathsf{A}\in\mathbb{A}_m}\osc\left(h-\mathbb{E}\left(h\vert\mathcal{B}_m\right),\left[\mathsf{A}\right]\right)\cdot \mu\left(\left[\mathsf{A}\right]\right)\notag\\
 &\leq \int\osc\left(h-\mathbb{E}\left(h\vert\mathcal{B}_m\right), B\left(\epsilon,x\right)\right)\mathrm{d}\mu\left(x\right)\notag\\
 &\leq \left|h-\mathbb{E}\left(h\vert\mathcal{B}_m\right)\right|_{\epsilon}\cdot\epsilon\notag\\
 &\leq\left(\left|h\right|_{\epsilon_0}+\left|\mathbb{E}\left(h\vert\mathcal{B}_m\right)\right|_{\epsilon_0}\right)\cdot \epsilon
 \leq 2\cdot \epsilon. \label{eq: f-Gn f}
\end{align}
The last inequality follows from the fact that for each cylinder set $C$ we have that 
$\esssup_{x\in C}h(x)\geq \esssup_{x\in C}\mathbb{E}\left(h\vert\mathcal{B}_m\right)\left(x\right)$
and $\essinf_{x\in C}h(x)\leq \essinf_{x\in C}\mathbb{E}\left(h\vert\mathcal{B}_m\right)\left(x\right)$
which yields $\left|\mathbb{E}\left(h\vert\mathcal{B}_m\right)\right|_{\epsilon_0}\leq \left|h\right|_{\epsilon_0}\leq 1$, 
since $h\in \mathsf{K}$.

In the following fix an arbitrary sequence $\left(h_n\right)\subset \mathsf{K}$ and a new sequence of function $(h_n^{\left(m\right)})_{n\in\mathbb{N}}\coloneqq\left(\mathbb{E}\left(h_n\vert\mathcal{B}_m\right)\right)_{n\in\mathbb{N}}$. 
For given $m\in\mathbb{N}$ we know that $(h_n^{\left(m\right)})$ is a sequence of bounded functions being piecewise constant on the same finite number of intervals.
Hence, there exists a subsequence $n\left(j,m\right)$ such that $(h_{n\left(j,m\right)}^{\left(m\right)})_{j\in\mathbb{N}}$ 
is a Cauchy sequence in $\mathcal{L}^1$ and thus converges. 
Additionally we might require the function $n:\mathbb{N}^2\to\mathbb{N}$ 
being such that for each $m\in\mathbb{N}$ we have $\left\{n\left(j,m+1\right)\colon j\geq 1\right\}\subset\left\{n\left(j,m\right)\colon j\geq 1\right\}$. 
If we set $\overline{n}\left(j\right)\coloneqq n\left(j,j\right)$, then for each $m\in\mathbb{N}$ 
the sequence $(h_{\overline{n}\left(j\right)}^{\left(m\right)})_{j\in\mathbb{N}}$ is a Cauchy sequence in $\mathcal{L}^1$. 
We can conclude that for all $\epsilon>0$ and $m_0\in\mathbb{N}$ there exists $J\in\mathbb{N}$ such that for all $m\geq m_0$ and all $j,\ell\geq \max\left\{m_0,J\right\}$ we have that 
\begin{align}
 \left|h_{\overline{n}\left(j\right)}^{\left(m\right)}-h_{\overline{n}\left(\ell\right)}^{\left(m\right)}\right|_1<\epsilon\label{eq: fnjk-fnlk}
\end{align}
are fulfilled at the same time. 

In the last steps we will apply \eqref{eq: f-Gn f} and \eqref{eq: fnjk-fnlk} to obtain
\begin{align*}
\left|h_{\overline{n}\left(j\right)}-h_{\overline{n}\left(\ell\right)}\right|_1
&\leq\left|h_{\overline{n}\left(j\right)}-h_{\overline{n}\left(j\right)}^{\left(m\right)}\right|_1+\left|h_{\overline{n}\left(j\right)}^{\left(m\right)}-h_{\overline{n}\left(\ell\right)}^{\left(m\right)}\right|_1+\left|h_{\overline{n}\left(\ell\right)}^{\left(m\right)}-h_{\overline{n}\left(\ell\right)}\right|_1
\leq5\epsilon
\end{align*}
which proves that $\left(h_{\overline{n}\left(j\right)}\right)$ is a Cauchy sequence and thus convergent in $\mathcal{L}^1$. 
Hence, each sequence has a convergent subsequence and $\mathsf{K}$ is compact.

Since $\mathsf{K}$ is compact, we consider in the following an arbitrary sequence $\left(h_n\right)$ with $\left\|h_n\right\|_{\epsilon_0}<1$, for all $n\in\mathbb{N}$. 
Since $\mathsf{H}$ is complete, see Lemma \ref{lem: H Banach algebra}, there exists a subsequence $n_j$ such that $h_{n_j}$ is Cauchy
and there also exists $h\in\mathsf{H}$ 
with $\left\|h\right\|_{\epsilon_0}<1$ such that $\lim_{j\to\infty}\left\|h_{n_j}-h\right\|_{\epsilon_0}=0$. 
By the proof of \eqref{Hennion 2} we have that $\left|L_f h_{n_j}-L_f h\right|_{1}=\left|L_f \left(h_{n_j}- h\right)\right|_{1}\leq \left| h_{n_j}-h\right|_{1}$ which tends to zero.
Setting $g\coloneqq L_f h$ yields \eqref{Hennion 4}.

Having proved \eqref{Hennion 1} to \eqref{Hennion 4} we can apply Lemma \ref{h it m d f} 
and obtain that $L_f$ is quasi-compact. 
Combining this with Lemma \ref{spec gap} and Lemma \ref{quasi comp to spec gap} yields that $L_f$ has a spectral gap. 
\end{proof}

\subsection{Proofs of Theorems \ref{thm: main thm reg var} and \ref{thm: Markov}}\label{subsec: prop D holds}
We will first prove the following key lemma:
\begin{lemma}\label{lem: Prop E implies Prop D}
 Assume that $(\X , \mathcal{B},\mu, \chi)$ fulfills Property $\mathfrak{F}$. 
 Then there exists $\epsilon_0\in (0,1)$ such that 
 $\left(\X ,\mathcal{B},\mu,\sigma,\mathsf{H},\left\|\cdot\right\|_{\epsilon_0},\chi\right)$
 fulfills Property $\mathfrak{D}$. 
\end{lemma}

\begin{proof}
 By Proposition \ref{prop: g measure}, $\mu$ is mixing and $\sigma$-invariant.
 By Lemma \ref{lem: H Banach algebra} $\mathsf{H}$ is a Banach algebra of functions which contains the constant functions 
 and fulfills $\left\|\cdot\right\|_{\epsilon_0}\geq \left|\cdot\right|_{\infty}$. 
 Proposition \ref{prop: spec gap} implies that there exists $\epsilon_0>0$ such that $\widehat{\sigma}$
 is a bounded linear operator with respect to $\left\|\cdot\right\|_{\epsilon_0}$
 and has spectral gap on 
 $\mathsf{H}$. 
 Hence, $\left(\X , \mathcal{B}, T,\mu,\mathsf{H},\left\|\cdot\right\|_{\epsilon_0}\right)$ fulfills Property $\mathfrak{C}$.
 Additionally, \eqref{eq: cond chi 1} and \eqref{eq: cond chi 2} imply that \eqref{C 1} and \eqref{C 2} are fulfilled for the same $\epsilon_0$
 giving Property $\mathfrak{D}$.
\end{proof}

\begin{proof}[Proof of Theorem \ref{thm: main thm reg var}]
 Theorem \ref{thm: main thm reg var} follows immediately by applying Lemma \ref{lem: Prop E implies Prop D} on
 \cite[Theorem 1.7]{kessebohmer_strong_2017}.
 This theorem gives the same statement as Theorem \ref{thm: main thm reg var} 
 with the condition on $\left(X,\mathcal{B},\mu,\chi\right)$ fulfilling Property $\mathfrak{F}$
 being replaced by the condition of $\left(\X ,\mathcal{B},\mu,\sigma,\mathsf{H},\left\|\cdot\right\|_{\epsilon_0},\chi\right)$
 fulfilling Property $\mathfrak{D}$.
\end{proof}

\begin{proof}[Proof of Theorem \ref{thm: Markov}]
We have to show that $(\X ,\mathcal{B},\mu,\chi)$ fulfills Property $\mathfrak{F}$.
Applying then Lemma \ref{lem: Prop E implies Prop D} and Theorem \ref{thm: asymp thm} immediately implies the statement of Theorem \ref{thm: Markov}. 

One can easily calculate that $\mu$ is a $g$-measure 
and $g\colon \X \to \X $ with 
$g\left(y\right)=\mu\left(x_1=y_1\vert x_2=y_2\right)$ the corresponding $g$-function
and it follows immediately that $\log g$ 
is Lipschitz continuous.

It remains to show 
\eqref{eq: cond chi 1} and \eqref{eq: cond chi 2}.
Let $[\mathfrak{a}_1]_n$ denote the $n$-cylinder $\left[\mathfrak{a}_1\ldots\mathfrak{a}_1\right]$.
For given $\epsilon\in (0,1)$ there exists $n\in\mathbb{N}_0$ such that $\epsilon\in \left[\mu([\mathfrak{a}_1]_{n}),\mu([\mathfrak{a}_1]_{n-1})\right)$.
This implies
\begin{align*}
 \osc\left(\upperleft{\ell}{}{\chi},B(\epsilon,x)\right)
 \leq \begin{cases}
       \ell &\text{ if }x\in [\mathfrak{a}_1]_{n}\\
       0&\text{ otherwise.}
      \end{cases}
\end{align*}
This implies 
\begin{align*}
 \frac{\int \osc\left(\upperleft{\ell}{}{\chi},B(\epsilon,x)\right)\mathrm{d}\mu}{\epsilon}
 &=\frac{\ell\cdot \mu\left([\mathfrak{a}_1]_{n}\right)}{\epsilon}\leq \ell.
\end{align*}
Since the choice of $\epsilon_0$ was arbitrary, \eqref{eq: cond chi 1} holds. 
Similarly, we have for the same choice of $\epsilon$ 
\begin{align*}
 \osc\left(\mathbbm{1}_{\{\chi\geq\ell\}},B(\epsilon,x)\right)
 \leq \begin{cases}
       1 &\text{ if }x\in [\mathfrak{a}_1]_{n}\\
       0&\text{ otherwise.}
      \end{cases}
\end{align*}
This implies 
\begin{align*}
 \frac{\int \osc\left(\mathbbm{1}_{\{\chi\geq\ell\}},B(\epsilon,x)\right)\mathrm{d}\mu}{\epsilon}
 &=\frac{\mu\left([\mathfrak{a}_1]_{n}\right)}{\epsilon}\leq \ell
\end{align*}
implying \eqref{eq: cond chi 2} since $\epsilon_0'$ was arbitrary.
\end{proof}


\begin{thebibliography}{KMS16}

\bibitem[Aar77]{aaronson_ergodic_1977}
J.~Aaronson.
\newblock On the ergodic theory of non-integrable functions and infinite
  measure spaces.
\newblock {\em Israel J. Math}, 27(2):163--173, 1977.

\bibitem[AN03]{aaronson_trimmed_2003}
J.~Aaronson and H.~Nakada.
\newblock Trimmed sums for non-negative, mixing stationary processes.
\newblock {\em Stochastic Process. Appl.}, 104(2):173--192, 2003.

\bibitem[BGT87]{bingham_regular_1987}
N.~H. Bingham, C.~M. Goldie, and J.~L. Teugels.
\newblock {\em Regular Variation}.
\newblock Cambridge University Press, Cambridge, 1987.

\bibitem[Bla97]{blank_discreteness_1997}
M.~Blank.
\newblock {\em Discreteness and continuity in problems of chaotic dynamics},
  volume 161 of {\em Translations of Mathematical Monographs}.
\newblock American Mathematical Society, Providence, RI, 1997.
\newblock Translated from the Russian manuscript by the author.

\bibitem[DF37]{doeblin_sur_1937}
W.~Doeblin and R.~Fortet.
\newblock Sur des cha{\^\i}nes {\`a} liaisons compl{\`e}tes.
\newblock {\em Bull. Soc. Math. France}, 65:132--148, 1937.

\bibitem[DV86]{diamond_estimates_1986}
H.~G. Diamond and J.~D. Vaaler.
\newblock Estimates for partial sums of continued fraction partial quotients.
\newblock {\em Pacific J. Math.}, 122(1):73--82, 1986.

\bibitem[GK11]{gyorfi_rate_2011}
L.~Gy\"orfi and P.~Kevei.
\newblock On the rate of convergence of the {S}t. {P}etersburg game.
\newblock {\em Period. Math. Hungar.}, 62(1):13--37, 2011.

\bibitem[Hae93]{haeusler_nonstandard_1993}
E.~Haeusler.
\newblock A nonstandard law of the iterated logarithm for trimmed sums.
\newblock {\em Ann. Probab.}, 21(2):831--860, 1993.

\bibitem[Hay14]{haynes_quantitative_2014}
A.~Haynes.
\newblock Quantitative ergodic theorems for weakly integrable functions.
\newblock {\em Ergodic Theory Dynam. Systems}, 34(2):534--542, 2014.

\bibitem[HH01]{hennion_limit_2001}
H.~Hennion and L.~Herv\'{e}.
\newblock {\em Limit Theorems for Markov Chains and Stochastic Properties of
  Dynamical Systems by Quasi-Compactness}.
\newblock Springer, Berlin, 2001.

\bibitem[HM87]{haeusler_laws_1987}
E.~Haeusler and D.~M. Mason.
\newblock Laws of the iterated logarithm for sums of the middle portion of the
  sample.
\newblock {\em Math. Proc. Cambridge Philos. Soc.}, 101(02):301--312, 1987.

\bibitem[HM91]{haeuler_asymptotic_1991}
E.~Haeusler and D.~M. Mason.
\newblock On the asymptotic behavior of sums of order statistics from a
  distribution with a slowly varying upper tail.
\newblock In {\em Sums, trimmed sums and extremes}, volume~23 of {\em Progr.
  Probab.}, pages 355--376. Birkh\"auser Boston, Boston, MA, 1991.

\bibitem[ITM50]{ionescu_theorie_1950}
C.~T. Ionescu-Tulcea and G.~Marinescu.
\newblock Theorie ergodique pour des classes d'operations non completement
  continues.
\newblock {\em Ann. of Math.}, 52(1):140--147, 1950.

\bibitem[Kea72]{keane_strongly_1972}
M.~Keane.
\newblock Strongly mixing {$g$}-measures.
\newblock {\em Invent. Math.}, 16:309--324, 1972.

\bibitem[KM92]{kesten_ratios_1992}
H.~Kesten and R.~A. Maller.
\newblock Ratios of trimmed sums and order statistics.
\newblock {\em Ann. Probab.}, 20(4):1805--1842, 1992.

\bibitem[KM95]{kesten_effect_1995}
H.~Kesten and R.~A. Maller.
\newblock The effect of trimming on the strong law of large numbers.
\newblock {\em Proc. London Math. Soc.}, s3-71:441--480, 1995.

\bibitem[KMS16]{MR3585883}
M.~Kesseb\"ohmer, S.~Munday, and B.~O. Stratmann.
\newblock {\em Infinite ergodic theory of numbers}.
\newblock De Gruyter Graduate. De Gruyter, Berlin, 2016.

\bibitem[KS17]{kessebohmer_strong_2016}
M.~Kesseb{\"o}hmer and T.~Schindler.
\newblock Strong laws of large numbers for intermediately trimmed sums of
  i.i.d. random variables with infinite mean.
\newblock {\em J. Theoret. Probab.}, pages 1--19, 2017 (online first), DOI: 10.1007/s10959-017-0802-0.

\bibitem[KS18]{kessebohmer_strong_2017}
M.~Kesseb{\"o}hmer and T.~Schindler.
\newblock Strong laws of large numbers for intermediately trimmed {B}irkhoff
  sums of observables with infinite mean.
\newblock {\em Stochastic Process. Appl.}, pages 1--45, 2018 (to appear),
DOI: 10.1016/j.spa.2018.11.015

\bibitem[Mal84]{maller_relative_1984}
R.~A. Maller.
\newblock Relative stability of trimmed sums.
\newblock {\em Z. Wahrsch. Verw. Gebiete}, 66(1):61--80, 1984.

\bibitem[Mor76]{mori_strong_1976}
T.~Mori.
\newblock The strong law of large numbers when extreme terms are excluded from
  sums.
\newblock {\em Z. Wahrsch. Verw. Gebiete}, 36(3):189--194, 1976.

\bibitem[Mor77]{mori_stability_1977}
T.~Mori.
\newblock Stability for sums of i.i.d. random variables when extreme terms are
  excluded.
\newblock {\em Z. Wahrsch. Verw. Gebiete}, 40(2):159--167, 1977.

\bibitem[Nak15]{nakata_limit_2015}
T.~Nakata.
\newblock Limit theorems for nonnegative independent random variables with
  truncation.
\newblock {\em Acta Math. Hungar.}, 145(1):1--16, 2015.

\bibitem[PP90]{parry_zeta_1990}
W.~Parry and M.~Pollicott.
\newblock {\em Zeta Functions and the Periodic Orbit Structure of Hyperbolic
  Dynamics}.
\newblock Soci{\'e}t{\'e} math{\'e}matique de France, Paris, 1990.

\bibitem[Sau00]{saussol_absolutely_2000}
B.~Saussol.
\newblock Absolutely continuous invariant measures for multidimensional
  expanding maps.
\newblock {\em Israel J. Math.}, 116(1):223--248, 2000.

\bibitem[Sch18]{schindler_observables_2018}
T.~Schindler.
\newblock Trimmed sums for observables on the doubling map.
\newblock {\em preprint: arXiv:1810.03223}, 2018.

\bibitem[Wal75]{walters_ruelle_1975}
P.~Walters.
\newblock Ruelle's operator theorem and {$g$}-measures.
\newblock {\em Trans. Amer. Math. Soc.}, 214:375--387, 1975.

\end{thebibliography}
\end{document}